\documentclass[10pt]{article}

\usepackage[utf8]{inputenc}
\usepackage[lined]{algorithm2e}
\usepackage{multirow,soul}
\usepackage{fixltx2e}
\usepackage{fancyvrb}
\usepackage{color}
\usepackage{enumitem}
\usepackage{datetime}
\usepackage{fullpage}

\usepackage[
    backend=biber,
    style=numeric,
    natbib=true,
    firstinits=true,
    url=false,
    doi=false,
    isbn=false,
    maxnames=5
]{biblatex}
\usepackage{amssymb,amsmath,amsthm,amsfonts,mathtools}
\usepackage{multicol}

\usepackage{latexsym}

\newcounter{saveenumi}

\providecommand{\vect}[1]{\boldsymbol{#1}}

\providecommand{\aee}{\text{ a.e.~}}

\renewcommand{\phi}{\varphi}

\providecommand{\ie}{i.e.~}
\providecommand{\eg}{e.g.~}

\DeclareMathOperator{\tr}{tr}

\providecommand{\id}{{{I}}}

\providecommand{\vect}[1]{\mathbf{#1}}

\providecommand{\tr}{\mathrm{tr}}

\providecommand{\ie}{i.e.~}

\providecommand{\eg}{e.g.~}

\newtheorem{theorem}{Theorem}[section]
\newtheorem{result}{Result}

\newtheorem{corollary}[theorem]{Corollary}
\newtheorem{lemma}[theorem]{Lemma}
\newtheorem{proposition}[theorem]{Proposition}

\newenvironment{remark}[1][Remark]{\begin{trivlist}
\item[\hskip \labelsep {\bfseries #1}]}{\end{trivlist}}

\providecommand{\nh}[2][]{\| #1 \|_{#2}}
\providecommand{\norm}[2][]{\left\| #1 \right\|_{#2}}

\providecommand{\tn}[1]{\left|#1\right|} 

\providecommand{\nltwo}[2][]{\left\| #1 \right\|_{#2}}

\newcommand{\N}{{\mathbb{N}}}
\newcommand{\R}{{\mathbb{R}}}

\providecommand{\pt}{\partial_3}

\providecommand{\pa}{\partial_\alpha}

\providecommand{\ua}{u_{\alpha}}
\providecommand{\va}{v{_\alpha}}
\providecommand{\vt}{v{_3}}

\providecommand{\eab}{e_{\alpha\beta}}
\providecommand{\eaa}{e_{\alpha\alpha}}

\providecommand{\ett}{e_{33}}
\providecommand{\eat}{e_{\alpha 3}}

\newcommand{\dx}{\mathrm{d}x}

\let\wto\rightharpoonup

\providecommand{\ke}[1]{\hat\kappa^\e(#1)}
\providecommand{\kev}{{\ke{\vect v}}}

\providecommand{\kttue}{\hat\kappa_{33}^\e(\vue)}

\providecommand{\kabue}{{\hat\kappa}_{\alpha\beta}^\e(\vue)}

\providecommand{\kebl}[1]{\kappa^\e(#1)}
\providecommand{\kttbl}[1]{{\kappa}_{33}^\e(#1)}
\providecommand{\kaabl}[1]{{\kappa}_{\alpha \alpha}^\e(#1)}

\providecommand{\kabbl}[1]{{{\kappa}}_{\alpha\beta}^\e(#1)}
\providecommand{\ktabl}[1]{{\kappa}_{3 \alpha}^\e(#1)}
\providecommand{\kbl}[1]{\kappa^\e(#1)}

\providecommand{\kevbl}{{\kebl{{\vect v}}}}

\providecommand{\kblue}{ \kappa^\e(\vue)}
\providecommand{\kttblue}{ \kappa_{33}^\e(\vue)}
\providecommand{\kaablue}{ \kappa_{\alpha \alpha}^\e(\vue)}

\providecommand{\kabblue}{ \kappa_{\alpha\beta}^\e(\vue)}
\providecommand{\ktablue}{ \kappa_{3 \alpha}^\e(\vue)}

\providecommand{\limkbl}{k}

\providecommand{\limkf}{\hat k}
\providecommand{\limkttf}{\hat k_{33}}

\providecommand{\limkabf}{\hat {k}_{\alpha\beta}}

\providecommand{\limkatf}{\hat k_{ \alpha 3}}

\providecommand{\limktt}{k_{33}}

\providecommand{\limkta}{k_{3 \alpha}}

\renewcommand{\O}{\Omega}
\renewcommand{\o}{\omega}
\providecommand{\G}{\Gamma}

\let\e\varepsilon

\newcommand{\Oe}{{\Omega_\e}}

\let\a=\alpha
\let\g=\gamma

\providecommand{\eab}{e_{\alpha\beta}}
\providecommand{\eaa}{e_{\alpha\alpha}}

\providecommand{\eat}{e_{\alpha 3}}
\providecommand{\ett}{e_{33}}
\providecommand{\eat}{e_{\alpha 3}}
\providecommand{\eit}{e_{i 3}}

\providecommand{\Qab}{Q_{\alpha\beta}}

\providecommand{\Qtt}{Q_{33}}
\providecommand{\Qat}{Q_{\alpha3}}

\newcommand{\urs}{\vect v^{\e,\eta}}

\newcommand{\Qrs}{Q^{\eta}}

\newcommand{\rsel}{\vect v^{\e}}
\newcommand{\rsms}{\vect w^{\e, \eta}}
\newcommand{\ors}{w^{\e, \eta}}

\providecommand{\ue}{u^\e}
\providecommand{\vue}{\vect{u}^\e}

\providecommand{\uea}{\ue_\alpha}
\providecommand{\uet}{\ue_3}
\providecommand{\ut}{{u_3}} 

\providecommand{\ve}{v^\e}

\providecommand{\Qtt}{Q_{33}}

\providecommand{\kabe}[1]{{\kappa}_{\alpha\beta}^\e(#1)}

\providecommand{\ke}[1]{\kappa^\e(#1)}
\providecommand{\kev}{{\ke{v}}}

\providecommand{\kebl}[1]{\kappa^\e(#1)}
\providecommand{\kttbl}[1]{{\kappa}_{33}^\e(#1)}
\providecommand{\kaabl}[1]{{\kappa}_{\alpha \alpha}^\e(#1)}

\providecommand{\kabbl}[1]{{{\kappa}}_{\alpha\beta}^\e(#1)}
\providecommand{\ktabl}[1]{{\kappa}_{3 \alpha}^\e(#1)}
\providecommand{\kbl}[1]{\kappa^\e(#1)}

\providecommand{\kevbl}{{\kebl{v}}}

\providecommand{\kblue}{\hat \kappa^\e(\ue)}
\providecommand{\kttblue}{\hat \kappa_{33}^\e(\ue)}
\providecommand{\kaablue}{\hat \kappa_{\alpha \alpha}^\e(\ue)}

\providecommand{\kabblue}{\hat \kappa_{\alpha\beta}^\e(\ue)}
\providecommand{\ktablue}{\hat \kappa_{3 \alpha}^\e(\ue)}

\providecommand{\limkbl}{\hat k}

\providecommand{\limkf}{k}

\providecommand{\limkttf}{k_{33}}

\providecommand{\limkabf}{{k}_{\alpha\beta}}

\newcommand{\muf}{\mu}
\newcommand{\mub}{\mu}

\newcommand{\laf}{\lambda}
\newcommand{\lab}{\lambda}
\newcommand{\mufr}{K_{\text{Fr}}^\e}

\newcommand{\Of}{{\O_f}}
\newcommand{\Ob}{{\O_b}}
\newcommand{\Oef}{{\O^\e_f}}
\newcommand{\Oeb}{{\O^\e_b}}

\providecommand{\ke}[1]{\kappa^\e(#1)}
\providecommand{\kev}{{\ke{v}}}

\providecommand{\kebl}[1]{\hat\kappa^\e(#1)}
\providecommand{\kttbl}[1]{\hat{\kappa}_{33}^\e(#1)}
\providecommand{\kaabl}[1]{\hat{\kappa}_{\alpha \alpha}^\e(#1)}

\providecommand{\kabbl}[1]{{\hat{\kappa}}_{\alpha\beta}^\e(#1)}
\providecommand{\ktabl}[1]{\hat{\kappa}_{3 \alpha}^\e(#1)}
\providecommand{\kbl}[1]{\hat\kappa^\e(#1)}

\providecommand{\kevbl}{{\kebl{v}}}

\providecommand{\kblue}{\hat \kappa^\e(\ue)}
\providecommand{\kttblue}{\hat \kappa_{33}^\e(\ue)}
\providecommand{\kaablue}{\hat \kappa_{\alpha \alpha}^\e(\ue)}

\providecommand{\kabblue}{\hat \kappa_{\alpha\beta}^\e(\ue)}
\providecommand{\ktablue}{\hat \kappa_{3 \alpha}^\e(\ue)}

\providecommand{\limkbl}{\hat k}

\providecommand{\limkf}{k}

\providecommand{\limkttf}{k_{33}}

\providecommand{\limkabf}{{k}_{\alpha\beta}}

\newcommand{\Qu}{\mathcal Q_U}
\newcommand{\Qx}{\mathcal Q_X}
\newcommand{\Qb}{\operatorname{co}\Qu}
\newcommand{\Qbb}{\mathcal Q_B }

\newcommand{\Fe}{I_\e}

\newcommand{\uk}{\vect u_k}

\newcommand{\infqx}{\inf_{Q\in H^1(\Ob,\Qx)}}

\newcommand{\infqbb}{\inf_{Q\in H^1(\Ob,\Qbb)}}
\newcommand{\infqu}{\inf_{Q\in H^1(\Ob,\Qu)}}

\makeatletter
\newsavebox\myboxA
\newsavebox\myboxB
\newlength\mylenA

\newcommand*\xoverline[2][0.75]{%
    \sbox{\myboxA}{$\m@th#2$}%
    \setbox\myboxB\null
    \ht\myboxB=\ht\myboxA%
    \dp\myboxB=\dp\myboxA%
    \wd\myboxB=#1\wd\myboxA
    \sbox\myboxB{$\m@th\overline{\copy\myboxB}$}
    \setlength\mylenA{\the\wd\myboxA}
    \addtolength\mylenA{-\the\wd\myboxB}%
    \ifdim\wd\myboxB<\wd\myboxA%
       \rlap{\hskip 0.5\mylenA\usebox\myboxB}{\usebox\myboxA}%
    \else
        \hskip -0.5\mylenA\rlap{\usebox\myboxA}{\hskip 0.5\mylenA\usebox\myboxB}%
    \fi}
\makeatother

\usepackage{boxedminipage}

\usepackage{listings}

\usepackage{minitoc}

\usepackage{ifpdf}
\usepackage{changes}
\definechangesauthor[color=orange]{ALB}
\definechangesauthor[color=blue]{PLC}

\newenvironment{system}%
{\left\lbrace\begin{array}{@{}l@{}}}%
{\end{array}\right.}

\title{Variational modelling of nematic elastomer foundations}
\author{Pierluigi Cesana, Andrés A. León Baldelli}

\begin{document}

\ifpdf
\DeclareGraphicsExtensions{.pdf, .jpg, .tif}
\else
\DeclareGraphicsExtensions{.eps, .jpg}
\fi

\maketitle

\begin{abstract}
We compute the $\Gamma$-limit of energy functionals describing mechanical systems composed of a thin nematic liquid crystal elastomer sustaining a homogeneous and isotropic elastic membrane. We work in the regime of infinitesimal displacements and model the orientation of the liquid crystal according to the order tensor theories of both Frank and De Gennes. We describe the asymptotic regime by analysing a family of functionals parametrised by the {vanishing} thickness of the membranes and the ratio of the elastic constants,
establishing that, in the limit, the system is represented by a two-dimensional integral functional interpreted as a linear membrane on top of a nematic active foundation involving an effective De Gennes optic tensor which allows for low order states. The latter can suppress shear energy by formation of microstructure as well as act as a pre-strain transmitted by the foundation to the overlying film.
 \end{abstract}

\tableofcontents

\let\oldurs\urs
\let\oldQrs\Qrs
\renewcommand{\urs}{\ensuremath{\vect v^\e}}
\renewcommand{\Qrs}{\ensuremath{Q^\delta}}

\newpage
\section{Introduction} 
\label{sec:introduction}
Nematic Liquid Crystal Elastomers (NLCEs) are a special class of soft Shape-Memory Alloys in which we observe not only shape-recovering induced by large deformations but also soft elastic deformation modes induced by the interplay of mechanical strain, order states, and optic microstructure. 
They are constituted of nematic molecules dissolved and cross-linked within the elastic matrix of a {polymeric} material. The polymeric backbone undergoes deformation as the nematic mesogens reorient driven by external stimuli, and conversely, a mechanical deformation of the structure leads to nematic reorientation in such a way that the director tends to be coaxial with the imposed principal stretches.
Activation of the shape-change mechanism can be triggered by means of external fields (such as electrostatic, magnetostatic as well as electromagnetic), mechanical constraints, and thermal frustration. In this last case, heating a sample of NLCE above a certain critical (Isotropic-to-Nematic transition) temperature $T_{IN}$, activates states of high entropic disorder until the material reaches the high symmetry phase of optical isotropy. From this stage, by cooling down below $T_{IN}$, a backward phase-transition is induced by breaking the high symmetry state thus enforcing order in the system.
The latter manifests as the LC molecules tend to spontaneously align parallel to each other in absence of other external stimuli. Newly formed states  of order  can be described by a unit vector field (the director) denoted by $\vect n$.
The transformation path from the isotropic high-symmetry phase to the nematic low-symmetry  state, which we refer to as Isotropic-to-Nematic phase transformation, is achieved via elastic and reversible deformations which typically show spontaneous uniaxial elongation along the director and contraction transversal to it, as imposed by volume conservation. Furthermore, the director is free to rotate with respect to the polymeric matrix which allows a large set of mechanical and optical instabilities possibly leading to pattern formation
and low order states.
Continuum modelling of LCEs in the framework of non-linear elasticity traces back to the work of the Cambridge group of Bladon, Warner and Terentjev. In a seminal paper~\cite{warner1996nematic}  an entropic elasticity model is derived which describes the novel phenomenon of soft elasticity, that is, the attainment of macroscopic uniaxial strains via formation of very fine shear-bands caused by the reorientation of the nematic director at a small scale.
Mathematically, this phenomenon has been rigorously studied in~\cite{desimone2000material} where the relaxation of the elastic energy of~\cite{warner1996nematic} has been obtained explicitly, thus  explaining the occurrence of nematic microstructure in terms of minimising sequences of energy functionals lacking lower-semicontinuity with the tools of the calculus of the variations.
Building upon this result, a substantial body of literature has appeared on the analytical modelling of NLCEs in both non-linear and linearised elasticity~\cite{desimone2000material,conti2002soft,desimone2002macroscopic,cesana2009strain-order,cesana2010relaxation,cesana2011nematic,cesana2011quasiconvex,agostiniani2011gamma,cesana2015effective,plucinsky2017microstructure}.

Investigation of reduced dimension theories with  emphasis on  membrane elasticity regime has surged in recent years inspired by a series of experiments exploiting the interaction of liquid crystals with  frustrations induced by slender geometries and topology~\cite{greco2013reversible,white2015programmable}.
In NLCE membranes, it has been observed that
nematic microstructure can determine or suppress elastic wrinkles, a mechanism which has suggested the design and realisation of wrinkle-free membranes made of NLCEs with potential application in aeronautics~\cite{cesana2015effective,plucinsky2017microstructure}.
As candidates for the design of electromechanical actuators~\cite{de2012liquid}, 
composite electro-active NLCE structures can be produced by either embedding carbon or ferroelectric particles as well as wires and tubes at nanometric scale, or by adding a conductive superficial layer to a NLCE film.
In the latter case, the thin bilayer structure couples nematic reorientation with {the} elasticity of the coating film, leading to the emergence of novel and rich phenomenology. 
Indeed, in this case, geometric constraints enter in the competition between nematic rigidity and material symmetries which are responsible of pattern formation.
Seeking {to generate} tunable micro-wrinkling patterns, the interaction of nematic microstructure and elasticity in free-standing electroactive bilayer membranes has been experimentally evidenced in~\cite{greco2013reversible}. Indeed, alongside the macroscopic mechanical deformation a periodic pattern of superficial micro-wrinkles is visible at the (free) surface of the LCE (see Figure~\ref{fig:experiment}).

Building up on the non-linear theory of~\cite{desimone2000material}, mathematical modelling of thin NLCEs mono-layers has been first accomplished  by~\cite{conti2002soft} in the context of planar membranes and later in more general geometries by~\cite{cesana2010relaxation}. 
Such modelling framework does account for formation of wrinkling and nematic microstructure, and can be developed toward the complete characterization of multi-layer structures (such as sandwiches and elastic foundations) to systematically account for the intrinsic length-scales introduced by the relative thicknesses and elastic moduli of the various layers.

 In this paper
we concentrate on understanding how  small scale features of the nematic order influence the macroscopic mechanical response of the structure, in the context of the mechanical actuation of thin bilayer NLCE/elastic membranes undergoing purely in-plane deformations (see Figure~\ref{fig:sketch} for a sketch of the three-dimensional system). This will allow {us} to characterise the macroscopic behaviour (in terms of its limiting energy) in the (vanishing thickness) thin film regime, and provide the energetic framework for the conception and analysis of representative experiments {as well as the} determination of solutions, both from the analytic and numerical standpoint. 

Our modelling approach is based on the theory of order tensors~\cite{de-gennes1993the-physics,virga1995variational}, whereby the nematic state variable, denoted by $Q$, accounts for both the local average direction of the LC molecules as well as their states of order.
In perfectly ordered systems, $Q$ is the Frank tensor, a uniaxial matrix with fixed non-zero eigenvalues.
In disordered systems, one may not be able to identify the exact direction of the molecules and rather  incorporate a description of the orientation of the molecules
in probabilistic terms. This has lead to the order tensor of De Gennes \cite{de-gennes1993the-physics}, still denoted by $Q$, within a model which accounts for full biaxial states and includes the case of isotropy, that is $Q=0$.
An extension of the order tensor theory for liquid crystals to the case of NLCEs, both for the Frank and the de Gennes tensor, has been proposed for 3D systems of nematic elastomers in~\cite{cesana2010relaxation,cesana2011nematic}.

We work in the regime of infinitesimal displacements, leaving the full nonlinear elasticity theory for future investigation. Although linearised elasticity has intrinsic limitations (see~\cite{bhattacharya2003microstructure}) it has been proved to approximate nonlinear theories of NLCEs in an energetic sense~\cite{agostiniani2011gamma}, and advantageously
allows to rigorously account for multiphysical phenomena (e.g., electric or magnetic fields and thermal stresses, curvature energies which are typical of liquid crystals) in the context of the order tensor theory thus {enabling}  a model extension which exploits linearity.

The {object} of this paper is the exact computation of the $\Gamma$-limit of a sequence of energy functionals describing a bi-layer composed of a NLCE membrane sustaining an {elastic} isotropic film. The result is two-fold: on one hand, we fully resolve the interaction of optical microstructure and the mechanical instabilities induced by the geometry and dimension reduction, by identifying all the available length-scales and their respective hierarchies. 
On the other hand we {qualitatively} characterise the minimiser of the $\Gamma$-limit in terms of the De Gennes order tensor thus proving that low-order states including perfect isotropy ($Q=0$) can be obtained \emph{effectively} via formation of microstructure while at the microscopic scale the order of the system is assumed to be fixed. This can be considered as the membrane version of the 3D result for NLCEs described in~\cite{cesana2011nematic,cesana2010relaxation}.

\medskip

The architecture of the paper is as follows. After concluding the introduction presenting (formally) the energy functional, 
in Section~\ref{sec:modelling_the_problem} we model the three-dimensional bilayer membrane, sketching the entire phase space which {describes all three-dimensional system as a function of} material parameters. We provide a non-technical statement and discussion of the main results in Section~\ref{sec:membrane}. Section~\ref{sec:proof_of_the_theorems} is dedicated to the proof of the theorems associated to the biaxial and uniaxial problems. Finally, in the Appendix we collect some tools used in the construction of the recovery sequence which apply in a general three-dimensional case.

\begin{figure}[tb]
    \centering
    \includegraphics[width=.65\textwidth]{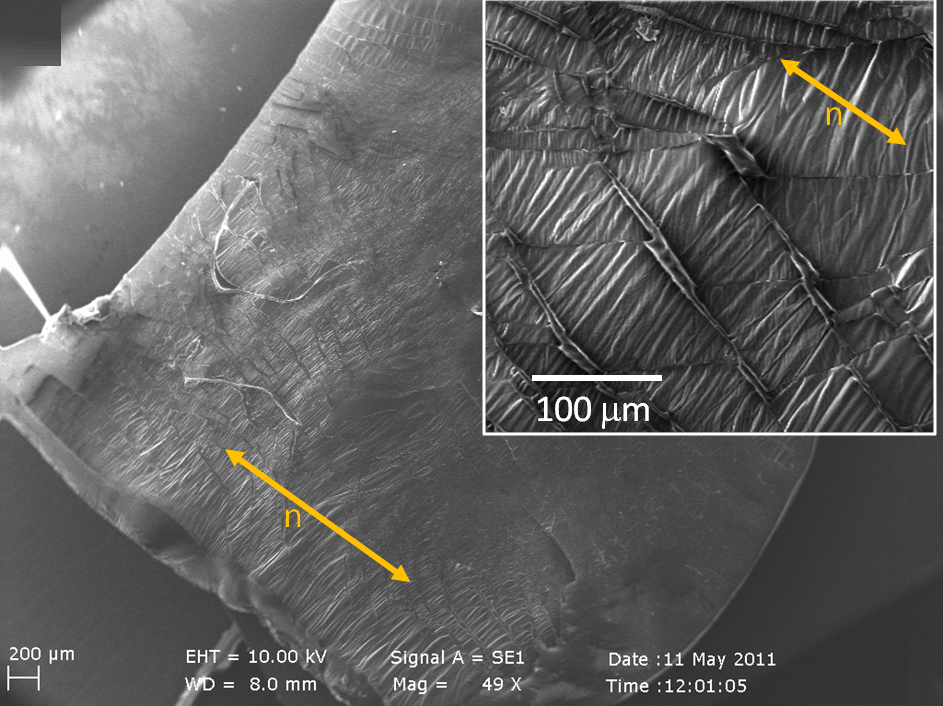}
    \caption{
Sample of a thin film of the conducting polymer poly(ethylenedioxythiophene):poly(styrene sulfonate) (PEDOT:PSS) deposited on a monodomain liquid crystal elastomer (LCE) film. The surface of the LCE/PEDOT:PSS composite shows optical microstructure and membrane wrinkling formed upon heating~\cite{greco2013reversible}.
The direction defining the director alignment  ($\vect n\otimes \vect n$) is shown in orange. The inset on the right shows a closer view of the surface in which uniaxial microwrinkles are formed and
organised in a patterned domain-like arrangement. The image is courtesy of F.~Greco, for a low resolution version see the cited article.
}
\label{fig:experiment}
\end{figure}

\subsection{The energy functional}

In the linearised setting, the bilayer membrane is described by the elastic strain $e(\vect v)$ (a symmetric matrix) and the local nematic order (the optic tensor) $Q$. The former is the symmetrised gradient of the displacement field $\vect v$, the latter is a quantity that accounts, at a mesoscopic scale, for the probability distribution of the orientation of nematic molecules and reflects material symmetries.
When both the degree of nematic order (including biaxiality) and the average direction of the molecules are taken into account as in de Landau-De Gennes theory~\cite{de-gennes1993the-physics}, the optic tensor $Q$ ranges in $\mathcal{Q}_B$, a convex and compact set (defined in Section \ref{sub:the_original_configuration}) which contains biaxial matrices and, amongst them, the null element.
On the other hand, considering the order of the system to be fixed so that the only variable describing nematic molecules is their common direction, the optic tensor is bound to the non-convex set $\Qu$ of uniaxial tensors (a subset of $\mathcal{Q}_B)$.

A typical energy of the structure under study couples the elastic film's deformation to that of the NLCE which, in turn, accounts for the interaction between the nematic monomers with the polymeric matrix as well as for the energetics of local reorientation at the level of nematic molecules. Formally, it can be written as follows
\begin{equation}
\begin{aligned}
\frac{1}{2} \int_{\Oef}
\underbrace{\left( 2\mu_f^\e \tn{e(\vect v)}^2 + \lambda_f^\e \tr^2(e(\vect v)) \right) }_{\text{film elastic energy density}} dx
+\frac{1}{2 }\int_{\Oeb}
\underbrace{\left( 2\mu_b^\e \tn{e(\vect v)-Q}^2 + \lambda_b^\e \tr^2(e(\vect v)) \right) }_{\text{elastic (nematic-prestrained)  density}} dx
+\frac{1}{2}\int_{\Oeb}\!\!\!\!\!\!\!\!\!\!\underbrace{\vphantom{\left( 2\mu_b^\e \tn{e(\vect v)-Q}^2\right)} \mufr \tn{\nabla Q}^2}_{\text{curvature energy density}}\!\!\!\!\!\!\!\!\!\!\!\!\! dx
\end{aligned}
\label{eqn:heuristicenergy}
\end{equation}

On a domain of small thickness, nematic reorientation plays the role of a pre-strain for the nematic elastomer attached to a rigid substrate and the overlying film. Its spatial variations are penalised by the curvature term and its energy is minimised when the optic tensor is coaxial with the elastic strain.
The nematic bonding layer undergoes elastic deformation, and although the latter is soft with respect to the film ($(\lambda_b^\e, \mu_b^\e)\ll (\lambda_f^\e, \mu_f^\e)$), its deformation is constrained by conditions (mismatching, in general) of place (at the interface with the substrate) and of continuity (at the film interface).
This results in mechanical frustration.
On the other hand, a non-trivial deformation of the film inducing deformation of the nematic bonding layer may be accommodated by nematic reorientation, {possibly attaining low energy states}.

\subsection{Notation and Preliminaries} 
Latin indices $i,j$ range within $\{1,2,3\}$ whereas Greek indices $\alpha, \beta$ range within $\{1, 2\}$.
We implicitly parametrise all vanishing quantities by a countable sequence.
We denote by $C$ uniform constants that may change from line to line.
With an overline bar $\xoverline v$ we indicate the thickness average of a map $v$, i.e. $\xoverline v(x') := |b-a|^{-1}\int_a^b v(x', x_3) \dx_3$.
All (but position) vectors are typeset in bold. 
We denote by $\odot $ the outer symmetrised  (tensor) product as in $(a \odot b)_{ij} = 1/2 (a_i b_j + a_j b_i)$ for $\vect a,\vect  b\in \R^3$.
We {employ standard} notation for function spaces, denoting by $L^2(\O;\R^n)$ and $ H^1(\O;\R^n)$ respectively, the Lebesgue space of square integrable functions on $\O$ with values in $\R^n$ and the Sobolev space of square integrable functions with values in $\R^n$ with square integrable weak derivatives on $\O$. 
We shall 
use the concise notation $L^2(\O)$ and $H^1(\O)$ whenever $n = 1$. The norm of a function $u$ in the normed space $X$ is denoted by $\nltwo[ u]{X}$.

\section{Modelling the problem} 
\label{sec:modelling_the_problem}

\subsection{The original configuration} 
\label{sub:the_original_configuration}

\begin{figure}[tb]
    \centering
    \includegraphics[width=.7\textwidth]{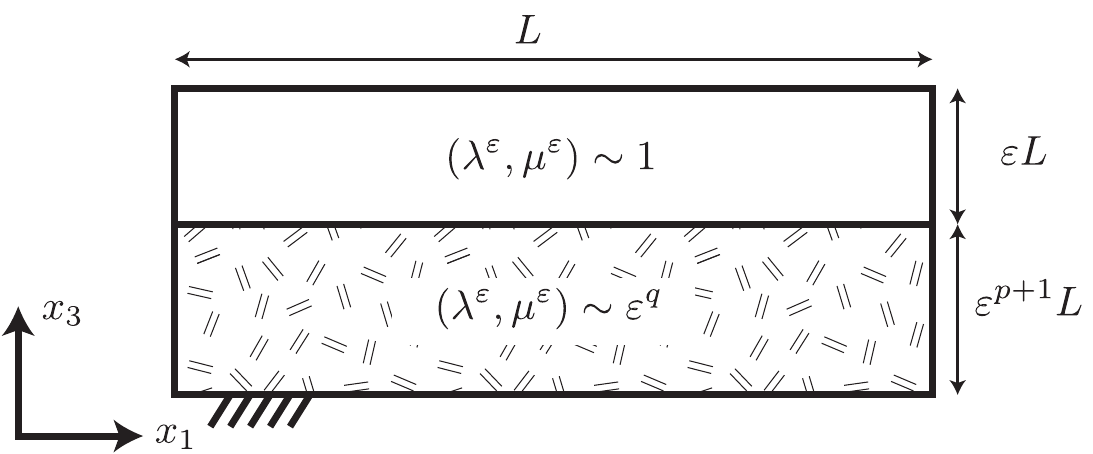}
    \caption{Sketch of the thin three-dimensional bilayer system. A soft nematic elastomer membrane is attached to a rigid substrate and supports a stiff isotropic elastic membrane.}
    \label{fig:sketch}
\end{figure}

\renewcommand{\b}{p}
\renewcommand{\a}{q}

Let $\Oe=\o\times (-\e^{\b+1}, \e)$ be a sufficiently smooth three-dimensional open set, \ie a domain, union of a thin linearly elastic film $\Oef=\o\times(0, \e)$, a soft nematic elastomer  $\Oeb=\o\times (-\e^{\b+1}, 0)$ and the interface $\o\times \{0\}$ separating them, where $\o\subset \R^2$. The domain is attached to a rigid substrate, see Figure~\ref{fig:sketch}. 
We focus on limit \emph{thin} systems requiring that $p+1>0$.
Any kinematically admissible displacement field $\vect v:\Oe\mapsto \R^3$ is required to satisfy the condition of place: $\vect v(x', -1)=0$ a.e. $x'\in\o$.
Accordingly, the space of admissible displacements is 
\[
    \mathcal{V}_\e := \left\{ \vect v\in H^1(\Oe; \R^3): \vect v(x', -1)=0 \text{ a.e. } x'\in \o  \right\}.
\]
We define the set of biaxial (De Gennes) tensors
\begin{eqnarray}\label{defiQQ}
\mathcal{Q}_{B}:=\Bigl\{ Q\in\R^{3\times 3}, \tr Q=0,Q=Q^T:\,\,
-\frac{1}{3}\leq\lambda_{\text{min}}(Q)\leq \lambda_{\text{max}}(Q)\leq
\frac{2}{3}
\Bigr\},
\end{eqnarray}
where $\lambda_{\text{min}}(Q)$ and $\lambda_{\text{max}}(Q)$ denote the
smallest and largest eigenvalue of the matrix $Q$. We remind that
$\mathcal{Q}_{B}$ is convex, closed and bounded. Then, we
introduce the set of uniaxial (Frank) tensors. This model uses only the eigenframe of $Q$ as the nematic state variable which is 
constrained to have eigenvalues $2/3,-1/3,-1/3$. Uniaxial tensors range
in the set
\begin{eqnarray}\label{phd012}
\mathcal{Q}_{U}:= \Bigl\{Q\in\mathcal{Q}_{B}:
\lambda_{\text{max}}(Q)=\frac{2}{3},\,\,
\lambda_{\text{min}}(Q)=-\frac{1}{3} \Bigr\}.
\end{eqnarray}
Notice that, since $\tr Q=0$, this suffices to describe the
spectrum of $Q$. It follows by the definition that
$\mathcal{Q}_{U}$ is a closed and non-convex set
and $\mathcal{Q}_{U}\subset\mathcal{Q}_{B}$.
Importantly, $\mathcal{Q} _B$ coincides with the convex envelope of
$\mathcal{Q}_{U}$.
{Let $O$ be any open set in $\R^n$, $n=2,3$. We define sets of tensors $L^2(O,\mathcal{Q}_X):=\{Q:O\to\mathcal{Q}_X \textrm{ a.e. in } O\}$ where $X$ stands for either $U$ (in the case of uniaxial order tensors) or $B$ (for biaxial tensors). Thanks to the definitions (\ref{defiQQ}) and (\ref{phd012}) it follows that both $L^2(O,\mathcal{Q}_U)$ and $L^2(O,\mathcal{Q}_B)$ are pointwise closed although, when endowed with the weak topology,  $L^2(O,\mathcal{Q}_B)$ is (weakly) closed by convexity while $L^2(O,\mathcal{Q}_U)$ is not. We then introduce the sets  $H^1(O,\mathcal{Q}_X)$, with $X=U$ or $B$ respectively, that is, the subsets of $H^1$-matrices with values in $L^2(O,\mathcal{Q}_X)$, with $X=U,B$.
}

We rewrite~\eqref{eqn:heuristicenergy} the total, unscaled, three-dimensional, mechanical energy of the system 
{for $\vect v\in \mathcal{V}_\e$ and $Q\in H^1(\Oeb, \Qx)$, where $X\in \{B, U\}$ indicates the nematic model at hand}

\newcommand{\Feh}{J_\e}
\begin{equation}
\label{eqn:defenergy}
    \frac{1}{2}\int_{\Oef} \widehat{W}_\e(\vect v) dx+\frac{1}{2}\int_\Oeb \left\{ \mufr|\nabla Q|^2+ W_{\e}(\vect v, Q) \right\}  dx.
\end{equation}
In the present linearly elastic setting, the energy densities $\widehat{W}_\e,W_\e$ are, respectively for the film and nematic layer, given by 
\begin{equation}
\label{eqn:energydensity}
     \widehat{W}_\e(\vect v):=2\muf_f^\e \tn{e(\vect v)}^2 + \laf_f^\e \tr^2(e(\vect v))     \quad\text{ and }\quad 
 W_\e(\vect v, Q):=2\mub_b^\e \tn{e(\vect v)-\g Q}^2 + \lab_b^\e \tr^2(e(\vect v)).
\end{equation}
Here, $\gamma$ is a fixed and positive constant that measures the interaction between $Q$ and the linearised strain $e(\vect v)=1/2 \left( \nabla \vect v +(\nabla \vect v)^T \right)$ which reads, in components, $(e(\vect v))_{ij} = 1/2 \left( \partial_j u_i + \partial_i u_j \right), i, j \in \{1, 2, 3\}$.
We henceforth fix, without loss of generality, $\gamma=1$.

The following parametric scaling law for the Lamé coefficients encodes the relative softness of the nematic elastic membrane with respect to the film
\[
    (\lambda^\e, \mu^\e)(x)=\begin{cases}
    (\laf_f, \muf_f), &x\text{ in }\Oef\\
    \e^\a(\lab_b, \mub_b), &x\text{ in }\Oeb 
\end{cases},\quad \text{with } \a>0.
\]
It is not restrictive for our analysis to set $\muf_f=\muf_b=\muf$ and $\lambda_f=\lambda_b=\lambda$
as we focus on relative order of magnitude, disregarding constants of order one.
Recalling that $E=\mu(3\lambda+2\mu)/(\lambda+\mu)$ and $\nu=\lambda/(2(\lambda+\mu))$, the above scaling law is equivalent to saying that, in terms of order of magnitude, the ratios of Young moduli and Poisson coefficients are {$E_b/E_f = \e^\a$ and $\nu_b/\nu_f=1$}.
Our analysis does not cover the physically relevant case of  incompressible nematic elastomers which could be accounted for by an additional singular perturbation to retrieve a limit linearised incompressibility constraint, e.g., by letting $\lambda_b$ to explode in $\Ob$.
We consider {that} the Lamé coefficients satisfy the classical inequalities established as consequences of three independent \emph{thought} experiments in homogeneous materials (see, eg. \cite{ciarlet1988three-dimensional})
\begin{eqnarray}\label{1711221148}
    \mu>0, \quad\lambda > 0 \quad (\text{and}\quad   2\mu + 3\lambda >0).
\end{eqnarray}
From a mathematical standpoint, the first and last inequalities represent necessary and sufficient conditions for coercivity of the energy functional.
The nematic curvature energy is characterised by $\mufr$, or Frank's stiffness, associated to the energy expense for changes in orientation of nematic molecules. Typical values of the ratio of nematic (Frank's) to elastic (Hooke's) rigidity can be assessed from experiments. Considering a thermotropic liquid crystal embedded in a backbone material ranging from a `weak rubber' to a `rather flexible polyethilene' network (see~\cite{warner2003liquid}), we can estimate
\[
    \hat\delta_\e:=\sqrt{\frac{\mufr}{\mub^\e}} \sim 10-100\cdot10^{-9}m.
\]
That the stiffness constant is related to the existence of a material length scale, heuristically, follows from a scaling argument considering the energetic cost of a transition layer between differently oriented domains.
From a phenomenological standpoint, experimental observation of nematic microstructure (see~\cite{kundler1995strain,bladon1994orientational}), hints that nematic rigidity imparts the system's smallest material length scale
which necessarily implies $\delta_\e \ll \e^{p+1},$
 $\e^{p+1}$ being the thickness of the nematic bonding layer.

The considerations upon material parameters, alongside the scaling of loading terms constitute the necessary assumptions upon the \emph{data} which ultimately determine the limit model.
We focus in the present note upon the inelastic load induced by reorientation of nematic molecules. 
We do not require additional assumptions on the scaling of $Q$ with respect to $\e$. {Consequently, although $Q$ is a loading term, its components do not scale with (i.e. are independent of) $\e$. Indeed, their magnitude in agreement with the underlying liquid crystal theory.}
In the uniaxial case, although the order of magnitude of spontaneous strains is known, their local orientation is a genuine unknown of the problem.
We refrain to explicitly consider loading terms within the film layer because they constitute a continuous perturbation of the unloaded energy functional, with respect to which our convergence result is stable.

\subsection{The rescaled configuration} 
\label{sub:the_rescaled_configuration}

The dependence of the energy functional with respect to the asymptotic parameter $\e$ is implicit in the geometry as well as in the constitutive laws. In order to render the relation explicit, we rescale dependent and independent variables.
For $(y', y_3)\in \Oef$ and $(x', x_3)\in \Of=\o\times (0, 1)$, let 
\begin{equation}
    \vect v(y', y_3) = \left( \e u_\alpha(x', \e x_3), u_3(x',\e x_3) \right)
    \label{eqn:scalingOf}
\end{equation}
and for $(y', y_3)\in \Oeb$ and $(x', x_3)\in \Ob=\o\times (-1, 0)$, let 
\begin{equation}
    \vect v(y', y_3) = \left( \e u_\alpha(x', \e^{\b+1} x_3), {\e^{-1}}u_3(x',\e^{\b+1} x_3) \right)
    \quad \text{ and }\quad Q(y', y_3) = Q(x', \e^{\b+1} x_3).
    \label{eqn:scalingOb}
\end{equation}
The above transformation, on the one hand, maps the $\e$-dependent domain $\Oe$ onto the fixed, unit domain $\O=\o\times (-1, 1)$.
On the other hand, the rescaling of displacements fixes  a specific relative magnitude of out-of-plane vs. in-plane displacements.
This may seem, up to this point, arbitrary. Recall, however, that in the case of purely elastic plates within the linear regime (cf., e.g., ~\cite[Vol 2]{ciarlet1988three-dimensional}), the scaling in the film~\eqref{eqn:scalingOf} is a \emph{result} and is implied, necessarily, by zero limit shear strains.
In the present case, different choices of the scaling of dependent variables would yield different limit models.
We stick to the above scaling because of the richness of phenomena that unfolds.
Ultimately, only comparison with experiments will dissolve the ambiguity within the family, vast indeed, of limit models based on different scaling assumptions.

Making use of the scalings of the dependent variables $\vect v$ in~\eqref{eqn:scalingOf} and~\eqref{eqn:scalingOb} in the energy of~\eqref{eqn:defenergy},  we compute the non-dimensional total energy (relative to the membrane energy of the film), at order $\e^{-3}$
    \begin{multline}
    \frac{\e^{-3}}{2\mu}\Feh(\vect u, Q) =\frac{1}{2}\int_\Of \left\{ \left( \e^{-2}\ett(\vect u) \right)^2  + \left| \eab(\vect u) \right|^2 + 2 \left| \frac{1}{2}\left( \pa \vt + \pt \e^{-1}\va \right) \right|^2 + \frac{\laf}{2\muf} (\eaa(\vect v)+\e^{-1}\ett(\vect v))^2 \right\} dx\\
    +\frac{1}{2}\e^{q}\e^{\b-2}
    \int_\Ob \left\{ \left( \frac{\ett(\vect u)}{\e^{\b+{2}}} -\Qtt \right)^2  + \left| \e\eab(\vect u) - \Qab\right|^2 + 2 \left| \frac{1}{2}\left( {\frac{\pa \ut}{\e}} + \frac{\e\pt \ua}{\e^{\b+1}} \right) -\Qat \right|^2\right\} dx\\
    +\frac{1}{2}\frac{\lab}{2\mu} \e^{q}\e^{\b-2}\int_\Ob \left\{ \left( \e\eaa(\vect u)+\frac{\ett(\vect u)}{\e^{\b+{2}}} \right) ^2  \right\}  dx
    +\frac{1}{2}\e^{q}\e^{\b-2}\int_\Ob 
    \delta_\e^2\left(\e^{2\b+2} |\nabla' Q|^2 +\tn{\pt Q}^2 \right)  dx.
\end{multline}
Note that the energy is written for (rescaled) displacements
\begin{equation}
    \label{eqn:admissiblespace}
    \vect u =(\ua, \ut)\in \mathcal{V} := \left\{ \vect v\in H^1(\O, \R^3): \vect v(x', -1)=0\quad \text{a.e. } x'\in \o  \right\}.
\end{equation}
In the energy above we have introduced the non-dimensional parameter {$\delta_\e^2:={\mufr}/\left(\mub^\e\e^{2\b+2} L^2 \right)=\hat\delta_\e^2/(\e^{2\b+2}L^2)$} where $\hat \delta_\e$ is the (three-dimensional) nematic length.
The nematic length scale $\delta_\e$ is infinitesimal. Indeed, with reference to the material parameters' range of the preceding section $\hat \delta_\e:=\sqrt{\mufr/\mub^\e}\sim 10-100\cdot 10^{-9}m$, and considering a bilayer consisting in two layers of equal thickness, e.g., {$\e L\sim 10^{-3}m$} we have $\delta_\e := \hat\delta_\e/(\e L)\sim 10- 100\cdot10^{-6}\ll 1$. 

In this regime the entire isotropic optic manifold can be exploited in order for the microstructure to attempt relaxing the elastic energy.
We shall see that, indeed, for each target in the limit space, there exists a three-dimensional, uniformly convergent, optimal sequence at scale $\eta_\e$, with $\delta_\e\ll \eta_\e\ll \e^{p+1}$, attaining the energy lower bound.

\let\p\xi
\let\q\zeta
Let us introduce the following non-dimensional quantities
\begin{equation}
\label{eqn:nondimparam}
    \xi:=\frac{\a + \b}{2}-1, \text{ and }
    \zeta:=\frac{\a - \b}{2}-1.	
\end{equation}
and write the non-dimensional scaled energy (per unit thickness and stiffness)
\begin{multline}
    \e^{-3}(2\mu)^{-1}\Feh(\vect u, Q)  =\\
    \frac{1}{2}\int_\Of \left( \left( \e^{-2}\ett(\vect u) \right)^2  + \left| \eab(\vect u) \right|^2 + 2 \left| \frac{1}{2}\left(  \frac{\pa \ut+\pt \ua}{\e} \right) \right|^2 + \frac{\laf}{2\muf} (\eaa(\vect u)+\e^{-2}\ett(\vect u))^2 \right) dx+ \\
    \frac{1}{2}\e^{2\p}\int_\Ob \left\{ \left( \e^{\q-\p}\frac{\ett(\vect u)}{{\e^2}} -\Qtt \right)^2  + \left| \e\eab(\vect u) - \Qab\right|^2 + 2 \left| \frac{1}{2}\left( {\e^{-1}}\pa \ut +  \e^{\q-\p}\pt \ua \right) -\Qat \right|^2\right\} dx\\
    +\frac{1}{2}\e^{2\p}\int_\Ob \frac{ \laf}{2\mu}\left\{  \left( \e\eaa(\vect u)+ \e^{\q-\p}\frac{\ett(\vect u)}{{\e^2}} \right) ^2 \right\} dx
    +\frac{1}{2}\e^{2\p}\int_\Ob \delta_\e^2 \left( \e^{{2\b+2}} \tn{\nabla' Q}^2 +  \tn{\pt Q}^2  \right)  dx.
\end{multline}
Here, $\xi$ and $\zeta$ identify the orders of magnitude of the nematic layer's membrane and shear energies, relative to the film's membrane energy, respectively. 
The geometric constraint $p+1>0$, identifying the regime of two-dimensional limit models, translates into $\zeta<\xi+1$.

A phase diagram, see Figure~\ref{fig:phasediag}, helps the mechanical interpretation of the different physical regimes.
Fixing one parameter, say $\b$, and allowing the other to vary, can be viewed equivalent as fixing the material mismatch and changing the aspect ratio. We start with some heuristic observations which we prove to be rigorous with the upcoming analysis.
The locus $\zeta=0$ represents three-dimensional systems in which the nematic shear energy (dominated by $\e^{2\zeta}\frac{1}{2}\int_\Ob |\pt \ua|^2dx$) is of the same order of the film's membrane energy ($\sim \tfrac{1}{2}\int_\Of |\eab(\vect u)|^2dx$), and hence the coupling. This regime identifies a purely membrane limit behaviour.
Along the line $\zeta=0$ the aspect ratio varies, the segment $\xi\in (-1, 0)$ represents three dimensional systems in which the film is thicker (in terms of order of magnitude) relatively to the nematic layer.
The half line $\xi >0$ identifies the opposite scenario. In the limit case $\xi=0$, both film thicknesses scale at the same rate.
The threshold at {$\e^{p+1}\sim\mathcal O(\delta_\e)$} is the limit of rigidity above which nematic reorientation with respect to $x_3$ is {energetically} unfavourable.
Above the line $\zeta=0$ the energy interaction is enriched by the possibility of out-of-place deformations, typically leading to plate-like models.
Heuristically, $\zeta$ identifies the dominant order of magnitude of the out-of-plane components of displacement (for a discussion, see~\cite{leon-baldelli2015on-the-asymptotic}).
Because a very rich phenomenology ensues from the physical case where coupling emerges between nematic shear and film membrane strain energy, we limit our study to the purely membrane scenario. 
Further,
for definiteness and in order not to weigh upon the notation, we focus on one representative case and henceforth fix without loss of generality $\xi=\zeta = 0$, or $q=2, p = 0$. All results of the present work can be adapted to the segment {$\zeta=0, -1<\xi<\mathcal{O}(\delta_\e)$}.

\begin{figure}[tb]
    \centering
    \includegraphics[width=.8\textwidth]{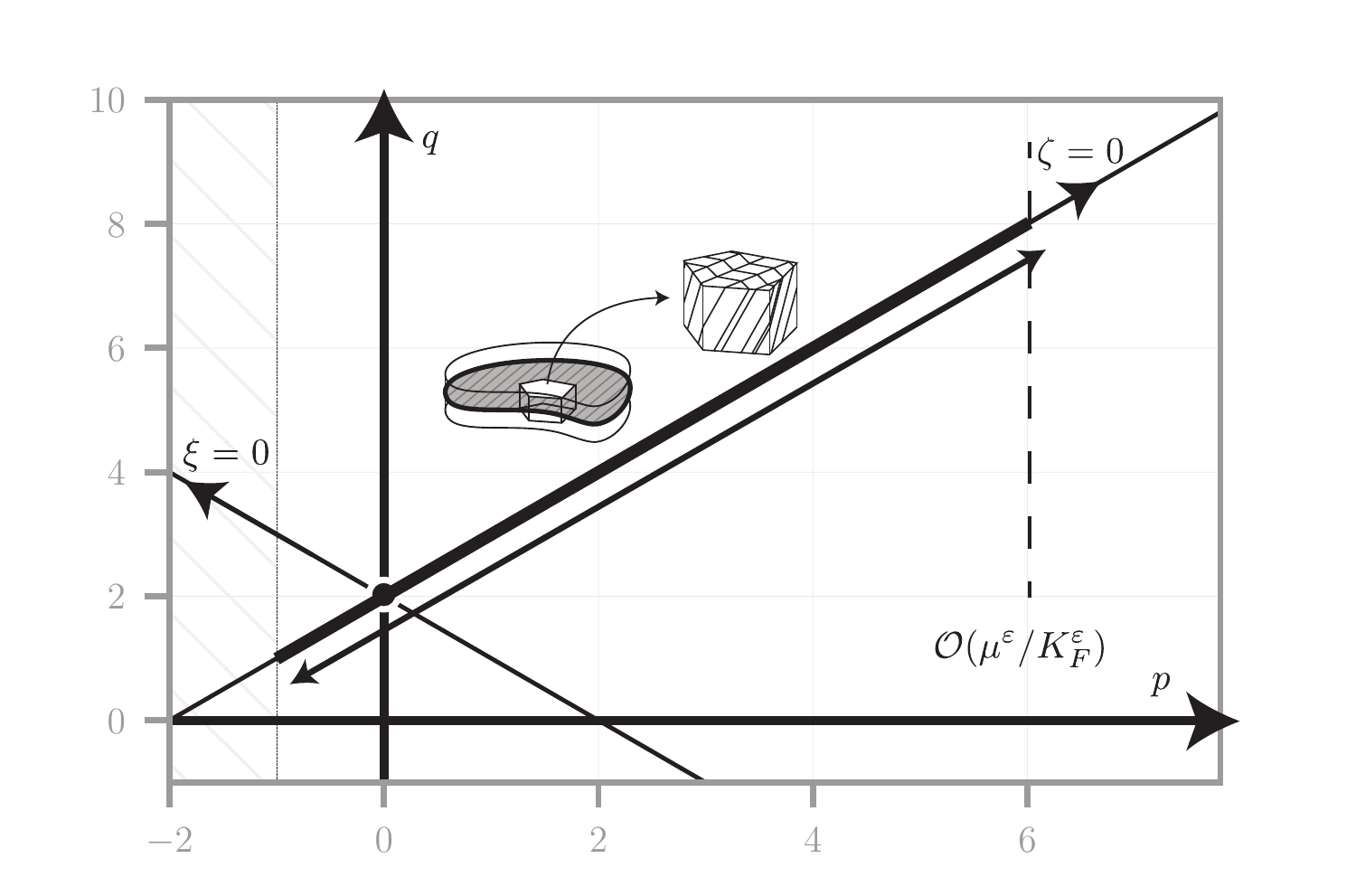}
    \caption{
    Phase diagram of three-dimensional systems. Here,
    the nondimensional parameters $\zeta=\frac{\a - \b}{2}-1$, $\xi=\frac{\a + \b}{2}-1$ identify material regimes. The locus $\zeta=0$ represent membrane-like systems with different aspect ratio, capable of relaxing shear energy by formation of microstructure. In the origin $\xi=\zeta=0$ both layers' thickness scale at the same rate. The vertical lines at {$p$ such that $\e^{p+1}\sim\mathcal O(\delta_\e)$} and $p=-1$ represent the limit of validity of the present asymptotic results.}
    \label{fig:phasediag}
\end{figure}

The expression of the energy above suggests the introduction of \emph{rescaled strains} which, as the analogue of elastic strains in~\eqref{eqn:energydensity}, allow for a compact expression of the energy and encode both the geometric separation of scales and the scaling laws of the elastic coefficients.
Let us define
\begin{equation}
        \kev=
    \begin{pmatrix}
    \eab(\vect v)&\e^{-1}\eat(\vect v)\\
    \text{sym.}&\e^{-2}\ett(\vect v)\\
    \end{pmatrix}\text{ in } \Of \quad \text{ and }\quad
    \kevbl=\begin{pmatrix}
    \e \eab(\vect v)&{\frac{1}{2}\left( \e^{-1}\pa \vt + \pt \va \right)} \\
    \frac{1}{2}\left( \e^{-1}\pa \vt + \pt \va \right)&{\e^{-2}}\ett(\vect v)\\
    \end{pmatrix}\text{ in } \Ob.
    \label{eqn:def_resc_str}
\end{equation}
We will show that the convergence of displacements descends from the compactness properties of rescaled strains.
\newcommand{\Feold}{I_\e}

Finally, we compute the total mechanical macroscopic energy of the entire structure by extracting the minimum, at small scale, among nematic orientations.
In this framework, the orientation of molecules is effectively imposed by the displacement, leaving us with a macroscopic energy functional depending only on the mechanical strain.
Analogous models of nematic elastomers based on the same minimisation procedure have been widely adopted in the literature for describing the engineering stress-strain response of NLCEs~\cite{cesana2011quasiconvex} and microstructure formation both in linearised \cite{cesana2010relaxation} and finite elasticity~\cite{desimone2002macroscopic} as well as membrane elasticity \cite{cesana2015effective,conti2002soft}.

We adopt a compact notation to denote the total energy functional in the uniaxial and biaxial order tensor theories. For $X=U$ or $B$, respectively, the energy reads
\renewcommand{\Fe}{E_\e}
\newcommand{\FXe}{E_{X,\e}}
\newcommand{\FBe}{E_{B,\e}}
\newcommand{\FUe}{E_{U,\e}}

\begin{equation}
\label{eqn:workingenergy}
E_{X,\e}(\vect u):=
 \infqx 
\widetilde E_{X,\e}(\vect u, Q)
\end{equation}
where
\begin{equation}
\label{eqn:workingenergy}
\widetilde E_{X,\e}(\vect u,Q):=
\begin{cases}
 \displaystyle\frac{1}{2} \int_\Of \left\{ \tn{\ke{\vect u}}^2 + \frac{\laf}{2\muf} \tr^2(\ke{\vect u})   \right\}dx+\\
 \displaystyle\frac{1}{2}
  \int_\Ob \left\{ \tn{\kebl{\vect u}-Q}^2 + \frac{\laf}{2\muf} \tr^2(\kebl{\vect u})  +   \delta_\e^2 \left( \e^2|\nabla' Q|^2+\tn{\pt Q}^2  \right)\right\} dx 
  \text{ if } (\vect u,Q)\in \mathcal V 
     \times {H^1}(\Ob;\Qx) \\
     \\
    +\infty  \qquad\qquad\qquad\qquad\qquad\qquad\qquad\qquad\qquad \text{otherwise in } L^2(\O,\R^3)\times {L^2}(\Ob,\R^{3\times 3}).
\end{cases}
\end{equation}

\section{The membrane regime} 
\label{sec:membrane}

The existence of minima for {the two models} $\widetilde E_{X,\e}(\vect u,Q)$ {with $X=B$ or $U$, respectively} at fixed $\e$, can be established by the direct method in the calculus of variations, considering admissible minimising sequences $(\vect u^k)\subset \mathcal V$ and $(Q^k)\subset H^1(\Ob,\Qx)$. The energy functionals are coercive in $\vect u$ (by Korn inequality, see~\cite{ciarlet2013linear}) and uniformly control $\nltwo[\nabla Q]{L^2(\Ob,\R^{3\times 3})}^2$. Further, {for $X=U$ and $B$,} the sets $L^2(\Ob,\Qx)$ are bounded and closed sets in the strong $L^2$ topology and therefore one can extract subsequences $Q^k $ which are weakly converging in $H^1$ and strongly converging in $L^2$ to a $Q\in H^1(\Ob,\Qx)$ {as $k\to \infty$}.
Finally, {$\widetilde{E}_{X,\varepsilon}$ } is lower semicontinuous with respect to the {weak $H^1(\O,\R^3)\times H^1(\Ob,\R^{3\times 3})$ topology.} 

In what follows we develop the the asymptotic analysis of the problems
\[
    \left\{ \inf {\FXe}(\vect u) \right\}_{\e>0}, \text{ as } \e\to 0, \text{ for } X=U, B.
\]

\subsection{Statement of the Problem and Main Results}

We attack the asymptotic {study} of the biaxial and uniaxial problems.
The first, a convex model, shows the main difficulties related to dimension reduction. A multilayer setting with similar scaling regimes yet for homogeneous non-active materials has been already discussed, in the case of elasticity and brittle elasticity, in~\cite{leon-baldelli2015on-the-asymptotic} and~\cite{babadjian2015reduced}. 
The proof of this asymptotic result is simpler than in the uniaxial case because, by convexity of $\Qbb$, no explicit relaxation is required for the variable $Q$.
The first result we obtain is the following:
\begin{result}[biaxial model, cf. Theorem~\ref{thm:biaxial}]
${\FBe}$ $\G$-converges to $E_0$ as $\e\to 0$ with respect to the $L^2(\O_f)$-strong topology, where $E_0$ reads
\begin{eqnarray}
E_0(\vect u):=
\begin{cases}
\displaystyle
\frac{1}{2}\int_\o \left\{ 
\frac{\laf}{\laf+2\muf}|\eaa(\xoverline{\vect u})|^2 + 2 |\eab( \xoverline{\vect u})|^2 
+\operatorname{dist}^2_{\mu,\lambda} \left( A(\overline{ \vect  u}),\Qbb \right)\right\} dx' 
&  \text{if } \vect u =(\xoverline{\vect u},0), \xoverline{\vect u} \in H^1(\omega,\R^2),  \\
+\infty & \text{otherwise in }L^2(\O,\R^{3}).
\end{cases} 
\end{eqnarray}
Here
\begin{eqnarray}
  A(\overline{ \vect  u})=\frac{1}{2}\left(
\begin{array}{ccc}
0 & 0 & \overline{u}_1  \\
0 & 0  & \overline{u}_2 \\
\overline{u}_1 & \overline{u}_2 & 0
\end{array} \right)
\end{eqnarray}
and
\begin{eqnarray}
dist_{\mu,\lambda}(A,\mathcal{Q}_B)=\inf_{Q\in\mathcal{Q}_B}\|A-Q\|_{\mu,\lambda},
\end{eqnarray}
where
\begin{eqnarray}\label{1711221146}
\|A\|_{\mu,\lambda}^2=|A_{\alpha\beta}|^2+2|A_{\alpha 3}|^2+\frac{\lambda}{\lambda+2\mu}|A_{33}|^2
\end{eqnarray}
for any $A\in \R^{3\times 3}$, $A=A^T$.

\end{result}

The limit model is a two-dimensional membrane undergoing purely in-plane deformations,
with an additional term accounting for the effective interaction with the soft nematic bonding layer. It is represented by the weighted squared distance of $A(\xoverline {\vect u})$, a symmetric tensor which measures the effective shear deformation of the nematic layer with respect to the set of $Q$-tensors. The weights are material-dependent parameters which are computed explicitly. Accordingly, this term is positive only when $A(\xoverline {\vect u})$ rests outside of $\Qbb$. In other terms, the nematic film behaves in the limit as an effective `active' elastic foundation, operating only when there is no admissible nematic order that accommodates the elastic deformation driven by the film.
Equation~\eqref{1711221146} defines a metric on symmetric matrices in view of the constitutive assumptions~\eqref{1711221148}. Notice that, thanks to the properties of the distance and by convexiy of $\mathcal{Q}_B$, there follows
\begin{equation}
\label{eqn:distance}
\int_{\omega}dist^2_{\mu,\lambda}(A(\overline{\vect u}),\mathcal{Q}_B) dx'=\inf_{\overline{Q}\in L^2(\omega,\mathcal{Q}_B)}
\int_{\omega}\left(\frac{\lab}{2\mub+\lab}\left( \xoverline Q_{33} \right)^2+2\tn{\frac{1}{2}\overline u_{\alpha}-\xoverline Q_{\alpha 3}}^2  + |\xoverline Q_{\alpha \beta}|^2\right)dx'.
\end{equation}
The fact above will be widely used in what follows in order to characterize the $\Gamma$-limits.

The proof of the convergence of the biaxial model is essentially a dimension reduction of a multilayer membrane, in the spirit of  \cite{babadjian2015reduced,leon-baldelli2015on-the-asymptotic}, with an unknown pre-strain in the nematic layer.
We prove this result, classically, by establishing lower- and a upper-bound inequalities; the first, entails key considerations of dimension reduction, the latter is an easy adaptation of~\cite{babadjian2015reduced}.
Once this question is settled, we tackle the more general uniaxial model.
The statement of our most comprehensive result is the following:
\begin{result}[uniaxial model, cf. Theorem~\ref{thm:uniaxial}]
${\FUe}$ $\G$-converges to $ E_0$ as $\e\to 0$  with respect to the $L^2(\O_f)$-strong topology, where $E_0$ reads as in Result~1.
\end{result}
In order to prove this result, we develop upon the arguments put forward for the biaxial model, the compactness and liminf lemmas extend straightforwardly.
The richness of the model at hand, however, is fully grasped in the construction of the recovery sequence for the limsup inequality. We analyse it thoroughly,  exhibiting the coupling between optic reorientation and elastic deformation.
Indeed, different quantities, associated with nematic and elastic phenomena, evolving at different scales along different directions, and converging in different topologies, are nonetheless coupled.
In order to fully relax the proposed problem, the construction requires calibrating infinitely fine microstructures at different spatial frequencies (in-plane vs. transverse), suitable smoothing of interfaces using very small `comfort zones', and constructing a microstructure rendering the elastic rigidity (that is, Hadamard jump conditions, or rank-1 connection of deformations) compatible with the symmetries of the isotropic nematic manifold.

The $\G$-limit is interpreted as the energy of a 2D elastic membrane on top of an in-plane `active' nematic foundation with a nematic microstructure which becomes effectively planar in the limit. The emergent microstructure is genuinely three-dimensional and allows full relaxation of the nematic manifold. Consequently, the nematic layer is able to accommodate with zero energy nontrivial elastic displacements belonging to the closed and convex set $\Qb\equiv\Qbb$.

\section{Proof of the Theorems} 
\label{sec:proof_of_the_theorems}

\subsection{Minimising sequences: compactness} 
\label{sub:compactness}

The choice of the topology is crucial in the analysis of the asymptotic limiting behaviour as different topologies may lead to very different asymptotic results which in turn may allow to embrace different physical phenomena. 
It ultimately determines compactness of sequences (including the recovery sequence). 
In this sense, the choice of the topology is a choice of mechanical modelling, likewise the choice of the functional setting and that of the energy functional.

We first show the proof of a Poincaré-type inequality.
\newcommand{\interval}{W}
\begin{lemma}[Poincaré-type inequality]
For any function $\vect u\in L^2(\O,\R^3)$ with $\pt \vect u\in L^2(W,\R^3)$,
with $\O=\o\times \interval$, $\o \subset \R^2$ open, bounded with Lipschitz boundary, $\interval=(a, b)\subset \R$ such that $u(\cdot, a)=0$ a.e. $x'\in \o$ we have $
    \nltwo[\vect u]{L^2(\O,\R^3)}\leq |\interval|\nltwo[\pt \vect u]{L^2(\O,\R^3)}.
$
\label{lem:poincare}
\end{lemma}

\begin{proof}
In what follows we additionally assume $\vect u(x',\cdot)$ is $C^1(W)$ for a.e. $x'\in\omega$. To obtain the desired result we operate by density.
Consider the inequality
\begin{multline}
|\vect u(x', x_3)| = |\vect u(x', x_3) - \vect u(x', 0)| = \left| \int_0^{x_3} \pt \vect u(x', s)ds\right| \leq \int_\interval |\pt \vect u(x', s)|ds \\
\leq \nltwo[\pt \vect u]{L^1(\interval,\R^3)}\leq |\interval|^{1/2} \nltwo[\pt \vect u]{L^2(\interval,\R^3)},
\end{multline}
where we have used the boundary condition in the first equality and Hölder's inequality in the last step, and $|W|=|b-a|$.
Then, computing
\begin{multline}
\nltwo[\vect u]{L^2(\O,\R^3)}=\left( \int_{\O} |\vect u|^2 dx\right)^{1/2} \leq \left( \int_{\O} |\interval| \nltwo[\pt \vect u]{L^2(\interval,\R^3)}^2 dx\right)^{1/2}=
\left(  |\interval|^2 \int_{\o}\nltwo[\pt \vect u]{L^2(\interval,\R^3)}^2 dx'\right)^{1/2}\\=   
 |\interval|\left( \int_{\o}\int_{\interval}|\pt \vect u|^2ds dx' \right)^{1/2} = |\interval| \nltwo[\pt \vect u]{L^2(\O,\R^3)}
\end{multline}
concludes the proof.
\end{proof}

We consider admissible minimising sequences $(\vect u^\e) \subset L^2(\O,\R^3)$ that leave the energy uniformly finite. 
The energy bound imparts, necessarily, growth properties to the scaled gradient of displacement. Then, using the Poincaré-type argument we infer their uniform boundedness in $L^2(\O,\R^3)$, so that there exists a compact set of $L^2(\O,\R^3)$ such that minimising sequences are compact therein, for all $\e$.

\renewcommand{\uk}{\vect u^\e}

\paragraph{Compactness.} 
\label{par:compactness_}

The compactness properties of sequences of equibounded energy lead to establishing their limit space and allow to retrieve necessary kinematic restrictions that apply in the limit.

Let us consider sequences $(\uk)\subset L^2(\O, \R^3)$ such that $\uk(\cdot, -1)=0 \aee\, \forall \e$ and $\uk \to \vect u$
strongly in $L^2(\Of,\R^3)$ as $\e\to 0$ (we consider here a countable index).
\newcommand{\FeX}{{\Fe}_{,X}}
We can assume that $\liminf_{\e\to0} \FXe(\uk)< \infty$, up to a subsequence.
Indeed, if $\liminf_{\e\to0} \FXe(\uk)=+\infty$, there is nothing to prove. Suppose that
\begin{equation}
    \FXe(\uk)\leq C,
    \label{eqn:energybound} 
\end{equation}
 for some $C>0$  independent of $\e$.
In particular, equiboundedness of the rescaled energy for minimising sequences implies
\[
    \nltwo[\ke{\uk}]{L^2(\Of)}\leq C\text{ and }\nltwo[\kebl{\uk}]{L^2(\Ob)}\leq C,
\]
so that there exist $\limkf\in L^2(\Of,\R^{3\times 3})$ and $ \limkbl\in L^2(\Ob,\R^{3\times 3})$ such that $\ke{\uk} \wto \limkf \text{ weakly in } L^2(\Of, \R^{3\times 3})$, $\kebl{\uk} \wto \limkbl \text{ weakly in } L^2(\Ob, \R^{3\times 3})$ as $\e\to 0$ (possibly up to extraction of a subsequence).
An application of Poincaré's inequality (Lemma~\ref{lem:poincare}) yields
\[
    \nltwo[\uk_3]{L^2(\O)}^2\leq\int_{\O} |\ett(\uk)|^2 dx \leq \int_\Of |\kttue|^2 dx+\int_\Ob |\kttblue|^2 dx\leq C{\e^2}\to 0
\]
showing that the component of the weak limit $u_3$ vanishes, that is, limit displacements are planar.
Recalling the definition of rescaled strains in the film and the energy bound \eqref{eqn:energybound} we get
\[
    \norm[\ett(\uk)]{L^2(\Of)}\leq C\e^2, \qquad
    \norm[\eat(\uk)]{L^2(\Of)}\leq C\e, \qquad
    \text{and }\norm[\eab(\uk)]{L^2(\Of)}\leq C
\]
which, using the expression of the energy and Korn's inequality (with boundary conditions) yields
\[
  \frac{1}{C}  \nh[\vue]{H^1(\Of,\R^3)} \leq \nltwo[e(\uk)]{L^2(\Of,\R^{3\times 3})}\leq \nltwo[\ke{\uk}]{L^2(\Of,\R^{3\times 3})}\leq C.
\]
Displacements are actually uniformly bounded in $H^1(\Of,\R^3)$, hence, up to a subsequence not relabelled, there exists $\vect u\in H^1(\Of,\R^3)$ such that $\uk \wto \vect u$ therein, as $\e\to 0$.
The energy bound yields an additional kinematic characterisation of the displacement field in the film. Indeed, \eqref{eqn:energybound} and \eqref{eqn:def_resc_str} imply $\eit(\uk)\to 0$ strongly in $L^2(\Of,\R^3)$. 
As a consequence, $\pt \ua = -\pa u_3 $ and since $\ut^\e\to 0$ strongly in $L^2(\Of, \R)$  $\pt \ua = -\pa u_3 \equiv 0$ pointwise a.e. in $\o$.
From this follows that $\ua = \ua(x')\in H^1(\Of, \R^2)$, that is: limit displacements are independent of $x_3$. They can be further characterised as an $H^1$ function defined on any section of the film. This identifies the limit space
\begin{equation}
    \label{eqn:limitspace}
    \mathcal{V}_0:=\left\{\vect v\in H^1(\o,\R^3): \vect v=(\overline {\vect v}, 0), \overline{\vect v}\in H^1(\o, \R^2) \right\}.
\end{equation}
Considering the quantities defined in the nematic bonding layer, we have
\begin{equation}
    \label{eqn:compactnessbl}
        \norm[\ett(\uk)]{L^2(\Ob)}\leq C\e\to 0, \qquad
    \norm[({\e^{-1}\pa \uk_3+\pt \uk_\alpha})]{L^2(\Ob)}\leq C, \qquad \text{ and }
    \norm[\e\eab(\uk)]{L^2(\Ob)}\leq C.
\end{equation}

Note that the only term of order zero in $\e$ is given by shear terms. Although scaled strains are bounded,
the energy does not fully control the symmetrised gradient of displacement due to coefficient $\e$ factoring the in-plane terms. 
However, compactness will prove sufficient 
to fully characterise the limit.

Because the sequence $(\uk)$ is bounded in $L^2(\Ob,\R^3)$ by Poincaré and the energy bound, and because of the first estimate above, we get $u^\e_3\to 0$ strongly in $L^2(\Ob)$ as well.
We infer that there exists $\vect u^*\in L^2(\Ob,\R^3)$ with ${ u}^{*}_3=0$  such that $\uk\wto \vect u^*\in L^2(\Ob,\R^3)$.
We are left to identify the planar components of limit displacements and their gradient.
Possibly passing to a further subsequence (not relabelled) 
 there holds
\[
\int_{\Ob}\vue \vect \phi dx\to \int_{\Ob} \vect u^*\vect \phi dx,\quad \text{ as } \e\to 0,\qquad \forall\vect \phi\in C^{\infty}_c(\Ob,\R^3).
\]
Also, since  $\partial_j \vect \phi\in L^2(\Ob, \R^{3\times 3})$
we have
\begin{eqnarray}\label{1611231640}
\int_{\Ob}{\vect u}^{\e}\partial_j\vect \phi dx\to \int_{\Ob} \vect u^*\partial_j\vect \phi dx,\quad 
\text{ as } \e\to 0,\qquad \forall \vect \phi\in C^{\infty}_c(\Ob,\R^3),
\end{eqnarray}
The relation (\ref{1611231640}) and the definition of distributional derivative imply that
\begin{eqnarray}\label{1611231659}
\int_{\Ob}
\partial_j \vue\vect \phi dx\to 
\int_{\Ob} 
\partial_j
\vect u^*\vect \phi dx,\quad \text{ as } \e\to 0,\qquad \forall \vect\phi\in C^{\infty}_c(\Ob,\R^3),
\end{eqnarray}
therefore $\nabla \vue\to\nabla \vect  u^*$ (in the sense of distributions) in $\mathcal{D}'(\Ob,\R^{3\times 3})$ (the dual of $C^{\infty}_c(\Ob,\R^3)$)
with $\nabla \vue\in L^2(\Ob,\R^3)$ and $\nabla \vect u^*\in \mathcal{D}'(\Ob,\R^{3\times 3})$.
The energy estimate for this particular subsequence  reads
\[
\e^2\int_{\Ob}|\nabla' \vue|^2 dx \leq \int_{\Ob}\e^2|\eab( \vect  u_k)|^2 dx\leq C 
\]
and yields, possibly up to a further sequence not relabelled, $\e\nabla' \vue\wto \vartheta$
for some $\vartheta\in L^2(\Ob,\R^{3\times2})$. In other words, we have
\begin{eqnarray}\label{1611231700}
\int_{\Ob}\e \partial_\alpha \vue \vect \phi dx\to
\int_{\Ob}\vartheta\vect \phi dx,\qquad \text{ as }\e\to 0,\qquad \forall \vect \phi\in L^{2}(\Ob,\R^3).
\end{eqnarray}
At this stage $\vartheta$ may depend on the particular subsequence.
Now recall formula (\ref{1611231659}) and that
relation \eqref{1611231700} holds, in particular, in the case $ \vect \phi\in C^{\infty}_c(\Ob,\R^3)$.
Multiply both sides of \eqref{1611231659}  by $\e$ to get
\begin{eqnarray}
\e\int_{\Ob}
\partial_\alpha \vue\vect \phi dx\to 
0,\quad \text{ as } \e\to 0,\qquad \forall \vect\phi\in C^{\infty}_c(\Ob,\R^3)
\end{eqnarray}
which characterises $\vartheta$ as a distribution, from \eqref{1611231700}, \ie $\vartheta\equiv 0$ for any subsequence.
Furthermore, by the fundamental theorem of calculus, the $0$ distribution is equivalent to the $0$ function (in the classical sense).\footnote{Differently from the strictly convex case of purely elastic multilayers (see \eg \cite{babadjian2015reduced}) where  $\vartheta$ is obtained upon minimisation,
here information upon limit strains is obtained by compactness, without further minimisation allowed.}
The identification of the limit $\vect u^*$ (which may well depend on its sub-sequence) in the whole $\Ob$ is not necessary because, as we will see, the limit energy only depends on the upper trace of limit displacements on $\o$ which, due to continuity, is identified with (the trace of) displacements in the film 
that determines the limit space $\mathcal{V}_0$.

\subsection{The biaxial model} 
\label{sub:the_biaxial_model}

\begin{theorem}
\label{thm:biaxial}
Let $\Ob=\o\times (-1, 0)$, $\Of=\o\times (0, 1)$, $\O=\o\times (-1, 1)$, $\o\subset \R^2$ open, bounded and with Lipschitz boundary, $\lambda, \mu >0$ and 
 ${\FBe}$ as defined in \eqref{eqn:workingenergy}, that is
\[
	{\FBe}(\vect u)=
	\begin{cases}
		 \displaystyle\frac{1}{2}\int_\Of \left(\tn{\ke{\vect u}}^2 + \frac{\laf}{2\muf} \tr^2(\ke{\vect u})   \right) dx
	 \\
	 +\inf_{Q\in L^2(\Ob,\Qbb)} \displaystyle\frac{1}{2}\int_\Ob \left(\tn{\kebl{\vect u}-Q}^2 + \frac{\lab}{2\muf} \tr^2(\kebl{\vect u}) + \delta_\e^2 \left(\e^2 |\nabla' Q|^2+\tn{\pt Q}^2  \right)\right) dx &\text{if } \vect u\in \mathcal V,\\
	+\infty \quad\quad\quad\quad\quad\quad\quad\quad\quad\quad\quad\quad\quad\quad\quad\quad\quad\quad\quad\quad\quad\quad\quad\quad\quad\quad\text{otherwise in }L^2(\O,\R^3).
	\end{cases}
\]
Then $\Gamma\hbox{-}    \lim_{\e\to 0} {\FBe}(\vue)=E_0(\vect u)$ with respect to the strong $L^2(\Of,\R^3)$ topology as $\e\to0$,
where 
\begin{eqnarray}\label{form0102}
E_0(\vect u)=
\begin{cases}
\displaystyle
\frac{1}{2}\int_\o \left\{ 
\frac{\laf}{\laf+2\muf}|\eaa(\xoverline{\vect u})|^2 + |\eab( \xoverline{\vect u})|^2 
+\operatorname{dist}^2_{\mu,\lambda} \left( A(\overline{ \vect  u}),\Qbb \right)\right\} dx' 
&  \text{if } \vect u  \in \mathcal V_0,  \\
+\infty & \text{otherwise in }L^2(\O,\R^{3}).
\end{cases} 
\end{eqnarray}
The spaces $\mathcal V$ and $\mathcal V_0$ are defined in \eqref{eqn:admissiblespace} and \eqref{eqn:limitspace}, respectively.

\end{theorem}

The lower semicontinuity inequality for the family of functionals requires to construct a lower bound $E_0$ for any converging sequence, based on ansatz-free, structural considerations on the energy, using optimality.
We prove the lower bound inequality in the following lemma.

\begin{lemma}[Lower bound]
\label{lem:lowerboundbi}
Let $\vect u \in \mathcal V_0$. For any sequence $(\vue)_{\e>0}\subset L^2(\O,\R^3):\vue \to \vect u$ strongly in $L^2(\Of, \R^3)$ we have
\newcommand{\liminflhs}{\liminf_{\e\to0}{\FBe}(\vue)}

\begin{equation}
\label{eqn:liminf}
{\liminf_{\e\to0}{\FBe}(\vue)}
    \geq\frac{1}{2}\int_\o \left\{ \frac{\laf}{\laf+2\muf}|\eaa(\xoverline{\vect u})|^2 + |\eab( \xoverline{\vect u})|^2 
+\operatorname{dist}^2_{\mu,\lambda} \left( A(\overline{ \vect  u}),\Qbb \right)\right\} dx'.
\end{equation}
\end{lemma}
\begin{proof}
Let us plug any admissible sequence $(\uk)\subset L^2(\O,\R^3)$. We can assume that, up to subsequences, the energy is uniformly bounded.
Let us compute
\begin{equation}
\begin{multlined}
\label{eqn:liminffilm}
\liminf_{\e\to0}
    \frac{1}{2}\int_\Of \left(\tn{\ke{\vue}}^2 + \frac{\laf}{2\muf} \tr^2(\ke{\vue}) \right) dx
    \geq 
        \frac{1}{2}\int_\Of
    \left(  |\limkf|^2
      +\frac{\laf}{2\muf} \tr^2(\limkf) \right) dx\\
          \geq
      \min_{k_{i3}} \left\{ \frac{1}{2}\int_\Of \left\{  \left(  \limkttf^2
      + \tn{\eab(\xoverline{\vect u})}^2 + {2} \tn{\limkatf }^2\right) + \frac{\laf}{2\muf} (\eaa(\xoverline{\vect u})+\limkttf)^2 \right\}dx : k_{i3}\in L^2(\Ob, \R^3)\right\}.
\end{multlined}
\end{equation}

The first inequality follows from the lower semicontinuity of the $L^2$ norm with respect to weak-$L^2(\Of, \R^{3\times3})$ convergence of the left hand side whereas for the second inequality we perform an explicit optimisation with respect to the third column of the (rescaled) strain tensor, \ie a convex pointwise optimisation across the thickness.
In addition, we have used the weak convergence of displacements, up to subsequences, $\vue\wto \vect u$ weakly in $H^1(\Of,\R^3)$ as $\e\to 0$ with $\vect u=(\xoverline {\vect u}, 0), \xoverline{\vect u}\in H^1(\o,\R^2)$ established by compactness. 
Because, by definition, $\kabue = \eab(\vue)$ we can characterise the limit of in-plane (scaled) strains $\limkabf=\eab(\xoverline{\vect u})$.

The convex optimisation above yields
\[
    \limkttf^*=\frac{-\laf}{\laf+2\muf}\eaa(\xoverline{\vect u}),\quad \limkatf^*=0.
\]
After plugging $k^*_{i3}$ in the integral on the right hand side of \eqref{eqn:liminffilm}, we integrate across the thickness owing to the invariance of limit displacements with respect to $x_3$. This establishes the lower bound inequality for the film, namely
\begin{equation}
\label{eqn:lbfilm}
    \liminf_{\e\to0}
    \frac{1}{2}\int_\Of \left(\tn{\ke{\vue}}^2 + \frac{\laf}{2\muf} \tr^2(\ke{\vue}) \right) dx
    \geq 
    \frac{1}{2}\int_\o \left(  \frac{\laf}{\laf+2\muf}|\eaa(\xoverline{\vect u})|^2 + |\eab( \xoverline{\vect u})|^2 \right)dx'.
\end{equation}
Note that the necessary relaxation with respect to the transverse (rescaled) strain in the general nonlinear case of plates and shells, see~\cite{le-dret1993le-modele}, boils down in the present linearly elastic case to a convex minimisation which can be computed explicitly.

Let us now tackle energy lower bound for the nematic layer
\begin{multline}
\label{eqn:liminfnematic}
\liminf_{\e\to 0}
\infqu \left\{
\frac{1}{2}\int_\Ob 
\left( \left( \kttblue-\Qtt  \right) ^2+2|\ktablue-\Qat|^2  + |\kabblue-\Qab|^2 \right) dx \right. \\
+
\left.
\frac{1}{2}\int_{\Ob}\left\{  \delta_\e^2\e^2|\nabla' Q|^2 +\delta_\e^2|\pt Q|^2  \right\} dx+ 
\frac{1}{2}\int_\Ob \frac{\lab}{2\mub} \left( \kttblue + \kaablue \right)^2 dx
\right\}\\
\geq  \liminf_{\e\to 0} 
\inf_{Q\in L^2(\O,\Qb)}
\frac{1}{2}\int_\Ob 
g^\e(\vue, Q)
dx,
\end{multline}
where $g^\e(\vect u, Q):=  \left( \kttbl{{\vect u}}-\Qtt  \right) ^2+2|\ktabl{ {\vect u}}-\Qat|^2  + |\kabbl{ {\vect u}}-\Qab|^2  +\frac{\lab}{2\mub} \left( \kttbl{ {\vect u}} + \kaabl{ {\vect u}} \right)^2$.

Now, notice the  functional above can be written in the form
\newcommand{\dummymatrix}{\kappa}
\[
    L^2(\O, \R^{3\times 3})\ni\kappa\mapsto\inf_{L^2(\O,\Qbb)}\int_{\O }\left\{ |\dummymatrix-Q|^2+ \frac{\lambda}{2\mu}(\tr\dummymatrix)^2 \right\}dx\equiv
\int_{\O }\left\{ 
\operatorname{dist}^2(\dummymatrix,\Qbb)+ \frac{\lambda}{2\mu}(\tr\dummymatrix)^2 \right\}dx.
\]
The last identity follows as in \eqref{eqn:distance} (see also~\cite{cesana2011nematic}). The functional has the structure of the euclidean distance of the matrix $\dummymatrix$ from a compact convex manifold, therefore it is a convex functional, hence lower semicontinuous {in the weak $L^2$ topology }with respect to the variable $\dummymatrix$.

Recalling the convergences established by compactness, that is $\kttblue\wto \limktt, \kabblue\wto 0, \ktablue\wto\limkta$ weakly in $L^2(\Ob)$ when $\e\to 0$, we can pass to the limit in $\e$ which yields
\begin{multline}
\liminf_{\e\to 0} 
\inf_{Q\in L^2(\O,\Qb)}
\frac{1}{2}\int_\o 
g^\e({\vue}, {Q})
dx\\
\geq 
 \inf_{Q\in L^2(\O,\Qb)}\frac{1}{2}\int_\Ob 
\left(  \left( \limktt-\Qtt \right)^2+2|\limkta-\Qat|^2  + |\Qab|^2 +\frac{\lab}{2\mub} (\limktt)^2  \right)  dx.
\end{multline}
Analogously to the film, we improve the lower bound upon pointwise minimisation with respect to $\limktt$ in $L^2(\Ob)$ which gives $\limktt^*=\frac{2\mub}{\lab+2\mub}\Qtt$. Substituting we obtain
\begin{multline}
\label{eqn:liminfnematictwo}
\liminf_{\e\to 0} 
\inf_{Q\in L^2(\O,\Qb)}
\frac{1}{2} \int_\Ob 
g^\e(\vue, Q)
dx\\\geq 
 \inf_{Q\in L^2(\Ob,\Qb)}
\frac{1}{2} \int_\Ob  \left\{ \left( \frac{\lab}{\lab+2\mub}\Qtt \right)^2+2|\limkta-\Qat|^2  + |\Qab|^2  + \frac{4\lab\mub^2}{(\lab+2\mub)^2}(\Qtt)^2  \right\} dx\\
= \inf_{Q\in L^2(\Ob,\Qb)}\frac{1}{2} \int_\Ob   \left( \frac{\lab}{\lab+2\mub}\left( \Qtt \right)^2+2|\limkta-\Qat|^2  + |\Qab|^2 \right)  dx\\
\geq
 \inf_{Q\in L^2(\Ob,\Qb)}\frac{1}{2} \int_\o \left(  \frac{\lab}{\lab+2\mub}\left( \Qtt \right)^2+2|\overline{k}_{3\alpha}-\Qat|^2  + |\Qab|^2 \right) dx'\\
=
 \inf_{\xoverline Q\in L^2(\o,\Qbb)}\frac{1}{2} \int_\o \left(  \frac{\lab}{\lab+2\mub}\left( \xoverline Q_{33} \right)^2+2|\overline{k}_{3\alpha}-\xoverline Q_{\alpha 3}|^2  + |\xoverline Q_{\alpha \beta}|^2 \right) dx'.
\end{multline}
The last inequality follows exploiting Cauchy-Schwartz in integrating  the convex integrand with respect to the variable $x_3$.
The last equality follows because the infimum computed 
in $L^2(\Omega_b,\mathcal{Q}_B)$ coincides with the infimum computed in 
 $L^2(\omega,\mathcal{Q}_B)$ because the function $\overline{k}_{3\alpha}$ is independent of $x_3$
 (consequently the matrix appearing in Eq. 
(\ref{eqn:liminfnematictwo})
has to be regarded as any biaxial tensor in 
 $\xoverline Q$ is a constant tensor with respect to $x_3$).
We can further characterise the limit $\overline k_{3\alpha}$.
First, let us estimate the in-plane derivatives of the sequence $\overline u_3^\e$. Recalling the definition of  $\ktabl{\vue}$, 
we can write
\begin{equation}
\begin{multlined}
    \label{eqn:boundinplanederiv}
        \nltwo[\frac{1}{2}{\e^{-1}}\pa \overline u_3^\e]{L^2(\o)} 
    \leq 
    \nltwo[\xoverline{\kappa}^\e_{3\alpha}(\vue) - \xoverline{Q}_{\alpha 3}]{L^2(\o)}
     +
     \nltwo[\frac{1}{2}\int_{-1}^0\pt\uea dx_3]{L^2(\o)}
     + \nltwo[\xoverline{Q}_{\alpha 3}]{L^2(\o)}
     \\
    =
        \nltwo[\xoverline{\kappa}^\e_{3\alpha}(\vue) - \xoverline{Q}_{\alpha 3}]{L^2(\o)}
     +
     \nltwo[\frac{1}{2}\uea(x', 0)]{L^2(\o)}
     + \nltwo[\xoverline{Q}_{\alpha 3}]{L^2(\o)}
     \leq C.
\end{multlined}
\end{equation}
We have used the triangle inequality and the crucial boundary condition 
${u}_\alpha^\e(x', -1)=0$ in the integration with respect to $x_3$.
The last inequality holds by the energy bound, because the biaxial set is bounded, and because the sequence $\uea(\cdot, 0)$ is bounded in $L^2(\o)$.
From the estimate above we get {$\pa \overline u_3^\e\to 0$ strongly in $L^2(\o)$.}
We can hence characterise 
\begin{equation}
\label{eqn:ktrealpha}
    \xoverline{\kappa}^\e_{3\alpha}(\vue)=\int_{-1}^0 \frac{1}{2}\left( \pa\uet+\pt \uea \right) dx_3 =\frac{1}{2}\left( \pa \overline u^\e_3+\uea(x',0) \right)  \to  \frac{1}{2}\xoverline u_\alpha(x')=\overline{k}_{3\alpha}, \text{ as }\e \to 0 \text{ in }L^2(\o,\R^2).
\end{equation}

Here, $\uea(x', 0)$ is the trace of $\uea$ at the interface $\o\times \{0\}$ and we have 
$\uea(x', 0)\to \xoverline u_\alpha^*(x')$ strongly, as $\e\to 0$, in $L^2(\o)$ owing to the trace theorem and the fact that $\uea(x', \cdot)$ is uniformly bounded in $H^1((-1, 0))$ a.e. $x'\in \o$.
By plugging into~\eqref{eqn:liminfnematictwo} the expression of $\xoverline\limkta$ in~\eqref{eqn:ktrealpha} and collecting inequality~\eqref{eqn:liminfnematic} we get

\begin{equation}
    \label{eqn:lowerboundnem}
\begin{multlined}
    \liminf_{\e\to 0}\infqbb 
\frac{1}{2}\int_{\Ob}
\left(  \left( \kblue -Q\right)^2
    +\frac{\lab}{2\mub} \tr\left( \kblue \right)^2
 + \delta_\e^2\e^2|\nabla' Q|^2 +\delta_\e^2|\pt Q|^2 \right)  dx 
\\
\geq \inf_{\overline Q\in L^2(\o,\Qbb)}\frac{1}{2}\int_\o  \left\{ \frac{\lab}{2\mub+\lab}\left( \xoverline Q_{33} \right)^2+2\tn{\frac{1}{2}\overline u_{\alpha}-\xoverline Q_{\alpha 3}}^2  + |\xoverline Q_{\alpha \beta}|^2\right\}dx'
    =
    \frac{1}{2}\int_\o
    \operatorname{dist}^2_{\mu,\lambda} \left( A(\overline{ \vect  u}),\Qbb \right) dx'.
\end{multlined}
\end{equation}
Combining \eqref{eqn:lbfilm} and \eqref{eqn:lowerboundnem} we obtain the desired lower bound inequality for the entire structure.
\end{proof}

We now exhibit a sequence that, for each target in the limit space, attains the lower bound.

\begin{lemma}[Upper bound]
\label{lem:upperboundbi}
There exists a sequence $(\vect v^\e)_{\e>0}\in L^2(\O, \R^3)$ such that $\vect v^\e\to \vect u=(\xoverline{\vect u}, 0)$ strongly in $L^2(\O,\R^3)$ and that, for every $\xoverline{\vect u}\in H^1(\o, \R^2)$
\[
    \limsup_{\e\to 0} {\FBe}(\urs)\leq E_0(\vect u). 
\]
\end{lemma}

\begin{proof}
Let us introduce an auxiliary two-variable functional, namely
for $(\vect u, Q)\in \mathcal V_0\times H^1(\Ob,\Qbb)$ with $\vect u(\cdot, -1)=0$ a.e. $x'\in \o$ define
\begin{multline}
\label{eqn:energyIbe}
 I_{B,\e}(\vect u,Q):=
\frac{1}{2}\int_\Of \left\{ \tn{\ke{\vect u}}^2 + \frac{\laf}{2\muf} \tr^2(\ke{\vect u})   \right\} dx
    \\+
    \frac{1}{2}\int_\Ob \left\{ \tn{\kebl{\vect u}-Q}^2 + \frac{\laf}{2\muf} \tr^2(\kebl{\vect u})
+  \delta_\e^2 \left( \e^2|\nabla' Q|^2+\tn{\pt Q}^2  \right) \right\}  dx
 \end{multline} and $I_{B,\e}(\vect u,Q):=+\infty$ otherwise in $L^2(\O,\R^3)\times L^2(\Ob,\R^{3\times 3})$.

If $\vect{u}\notin \mathcal{V}_0$ 
then there is nothing to prove. 
Hence we assume $\vect u = (\xoverline{\vect u}, 0), \xoverline{\vect u}\in H^1(\o,\R^2)$ in what follows and define the recovery sequence
\[
    v_\alpha^\e(x):=
    \begin{cases}
        \xoverline u_\alpha(x'), &\text{ if }x\in \Of\\
        (x_3+1)\xoverline u_\alpha(x'), &\text{ if }x\in \Ob\\
    \end{cases}, \quad
    v_3^\e(x):=
    \begin{cases}
        \e^2 (x_3 \hat h^\e(x')+{h^\e(x')}), &\text{ if }x\in \Of\\
        {\e^2} (x_3+1) h^\e(x'), &\text{ if }x\in \Ob\\
    \end{cases}
\]
\[
   \text{ and } Q^\e\equiv \xoverline Q\in H^1(\o,\Qbb),
\]
where 
\begin{equation}
    \label{eqn:auxconverglimsup}
C^{\infty}_c(\o)\ni \hat h^{\e}(x')\to \frac{-\laf}{\laf+2\muf} \eaa(\xoverline {\vect u})    \quad\text{ and }\quad
C^{\infty}_c(\o)\ni h_{\e}(x')\to \frac{2\mub}{\lab+2\mub}\xoverline Q_{33}, \text{ as }\e\to 0
\end{equation}
with control on the gradient, e.g. 
$\lim_{\e\to 0}\e\|\nabla' \hat h^{\e}\|_{L^2(\o)}=0$ and $\lim_{\e\to 0}\e\|\nabla'  h^{\e}\|_{L^2(\o)}=0$. 
The in-plane components of displacement are fixed with respect to $\e$, whereas the out-of-plane terms at order $\e^{2}$ achieve optimality (recall the explicit minimisations while establishing the lower bound).
{Note that continuity (of scaled displacements) requires that out-of-plane components of the recovery sequence be of the same order of magnitude in both $\Of$ and $\Ob$.}
As for the order tensor note that, with a slight abuse of notation, we consider the constant extension across the thickness of $\xoverline Q$.

Plugging the recovery sequence in the expression of the energy{~\eqref{eqn:energyIbe}}, we can compute
\begin{equation}
\label{eqn:limsupbizero}
    \begin{aligned}
        I_{B,\e}(\vect v^\e,Q^\e)=&
    \frac{1}{2}\int_\Of \left( \hat h^2_\e + \tn{\eab(\xoverline {\vect u})}^2 + 2\tn{\frac{1}{2}\e {(x_3\nabla' \hat h^\e+\nabla' h^\e)}}^2 + \frac{\laf}{2\muf}\tn{\hat h^\e + \eaa(\xoverline {\vect u})}^2 \right) dx\\
    &\quad+ \frac{1}{2}\int_\Ob 
    \left(
     (h^\e-\xoverline Q^\e_{33})^2 + \tn{\e (x_3+1) \eab(\xoverline{\vect u}) - \xoverline Q^\e{\alpha \beta}}^2 + 2 \tn{\frac{1}{2}\left( \xoverline u_\alpha + \e(x_3+1)\nabla' h^\e\right) - \xoverline Q_{\alpha 3}  }^2 \right)dx\\
     &\quad +\frac{1}{2}\int_\Ob
     \frac{\lab}{2\muf} \left( h^\e + \e(x_3+1)\eaa(\xoverline{\vect u}) \right)^2 +
         \delta_\e^2
         \left( \e^2|\nabla' \xoverline Q|^2+\tn{\pt \xoverline Q}^2  \right) dx.
    \end{aligned}
\end{equation}
We now pass to the limit in~\eqref{eqn:limsupbizero} using the convergences~\eqref{eqn:auxconverglimsup} established for $\hat h^\e, h^\e$ and noting that, because $\delta_\e\to 0 $ as $\e\to 0$ and $\xoverline Q\in H^1(\o, \Qbb)$, we have  $\int_\Ob  \left( \e^2\delta_\e^2\tn{\nabla' \xoverline Q}^2 + \delta_\e^2\tn{\pt \xoverline Q}^2 \right) dx\to 0$, and we are left with
\begin{multline}
\label{eqn:limsupbi}
   \limsup_{\e\to 0}I_{B,\e}(\vect v^\e,Q^\e) = 
   \frac{1}{2}\int_\Of \left( \frac{\laf}{\laf+2\muf} \eaa(\xoverline{\vect u})^2
+ \tn{\eab(\xoverline{\vect u})}^2 \right)dx 
\\
+\frac{1}{2}\int_\Ob \left( 
 \frac{\lab}{\lab + 2\mub} \xoverline Q_{33}^2 + 2 \tn{\frac{\overline{u}_\alpha}{2}-\xoverline Q_{\alpha 3}}^2+\tn{\xoverline Q_{\alpha \beta}}^2 \right) dx.
\end{multline}
Now, the following holds
\begin{equation}
\label{eqn:limsupbiprime}
\begin{multlined}
    \limsup_{\e\to 0}\left( \inf_{\overline Q \in H^1(\Ob; \Qbb)} I_{B,\e}(\vect v^\e,\xoverline Q) \right)\leq \limsup_{\e\to 0}I_{B,\e}(\vect v^\e,\xoverline Q)\\
= 
   \frac{1}{2}\int_\o \left( \frac{\laf}{\laf+2\muf} \eaa(\xoverline{\vect u})^2
+ \tn{\eab(\xoverline{\vect u})}^2 
 +\frac{\lab}{\lab + 2\mub} \xoverline Q_{33}^2 + 2 \tn{\frac{\overline{u}_\alpha}{2}-\xoverline  Q_{\alpha 3}}^2+\tn{\xoverline Q_{\alpha \beta}}^2 \right) dx'.
\end{multlined}
\end{equation}
Here the first inequality is trivial and the equality follows integrating the right hand side of \eqref{eqn:limsupbi} across the thickness as none of the variables depends on $x_3$.
Then, observe that any tensor in $L^2(\o,\Qbb)$ 
can be approximated by density with a sequence of biaxial tensors in $H^1(\o, \Qbb)$ 
in the strong topology of $L^2(\o,\R^{3\times 3})$.
This fact follows from convolution and the convexity and compactness of the set $\Qbb$.
Since the functional on the second line of~\eqref{eqn:limsupbiprime} is continuous in the strong topology of $L^2(\o,\R^{3\times 3})$, the same inequality holds for $\xoverline Q\in L^2(\o, \Qbb)$ as well.
Finally, the chain of inequalities thus obtained still holds if we replace the generic tensor $\overline{Q}$ in $L^2(\o,\Qbb)$ with the (unique) map $\overline{Q}^*$ which minimises the functional on the second line of~\eqref{eqn:limsupbiprime} in $L^2(\o,\Qbb)$ for fixed $\xoverline{\vect u}$.

This leaves us with
\begin{equation}
\begin{multlined}
    \limsup_{\e\to 0}\left( \inf_{\overline Q \in H^1(\Ob, \Qbb)} I_{B,\e}(\vect v^\e,\xoverline Q) \right)\\
    \leq 
    \frac{1}{2}\int_\o \left( \frac{\laf}{\laf+2\muf} \eaa(\xoverline{\vect u})^2
+ \tn{\eab(\xoverline{\vect u})}^2 
 +\frac{\lab}{\lab + 2\mub} (\xoverline{Q}^*_{33})^2 + 2 \tn{\frac{\overline{u}_\alpha}{2}-\xoverline{Q}^*_{\alpha 3}}^2+\tn{\xoverline {Q}_{\alpha \beta}^*}^2 \right) dx'\\
 =\frac{1}{2}\int_\o \left( \frac{\laf}{\laf+2\muf} \eaa(\xoverline{\vect u})^2
+ \tn{\eab(\xoverline{\vect u})}^2 +
 \operatorname{dist}^2_{\mu,\lambda} \left( A(\xoverline{ \vect  u}),\Qbb \right) \right)dx'
\end{multlined}
\end{equation}
as required.

\end{proof}

\subsection{The uniaxial model}\label{1711071646}
We now tackle the uniaxial case. The main difference with respect to the biaxial tensor model is due to the necessity to explicitly construct an optimal microstructure that allows relaxation of the energy and convexification of the nematic manifold.
We establish the following result.

\label{sub:the_uniaxial_case}

\begin{theorem}\label{1710312258}
Let $\Ob=\o\times (-1, 0)$, $\Of=\o\times (0, 1)$, $\O=\o\times (-1, 1)$, $\o\subset \R^2$ open, bounded, and with Lipschitz boundary, $\lambda, \mu >0$ and 
 $\FUe$ as defined in \eqref{eqn:workingenergy}, that is
\[
	{\FUe}(\vect u):=
	\begin{cases}
    \displaystyle\frac{1}{2}\int_\Of \left( \tn{\ke{\vect u}}^2 + \frac{\laf}{2\muf} \tr^2(\ke{\vect u})   \right) dx
     \\
     +\infqu\displaystyle \frac{1}{2}\int_\Ob \left(\tn{\kebl{\vect u}-Q}^2 + \frac{\lab}{2\muf} \tr^2(\kebl{\vect u}) +  \delta_\e^2 \left( \e^2|\nabla' Q|^2+\tn{\pt Q}^2 \right)\right) dx  
 &\text{if } \vect u\in \mathcal V\\
		+\infty \quad\quad\quad\quad\quad\quad\quad\quad\quad\quad\quad\quad\quad\quad\quad\quad\quad\quad\quad\quad\quad\quad\quad\quad\text{otherwise in }L^2(\O,\R^3).
	\end{cases}
\]
Then 
$\Gamma\hbox{-}    \lim_{\e\to 0} {\FUe}(\vue)=E_0(\vect u)$,
where the limit is computed in the strong $L^2(\Of,\R^3)$ topology and 
\begin{eqnarray}\label{form0102}
E_0(\vect u)=
\begin{cases}
\displaystyle{ \frac{1}{2}\int_\o \left\{ \frac{\laf}{\laf+2\muf}|\eaa(\xoverline{\vect u})|^2 +  |\eab( \xoverline{\vect u})|^2 
+\operatorname{dist}^2_{\mu,\lambda} \left( A(\xoverline{ \vect  u}),\Qbb \right)\right\} dx'}
&  \text{if } \vect u \in \mathcal{V}_0\\
+\infty & \textrm{otherwise in }L^2(\O, \R^{3 }).
\end{cases}
\end{eqnarray}
Moreover, 
let $(\vect u_{\e})\subset L^2(\O,\R^3)$ be such that
$
\lim_{\e\to 0} \FUe(\vect u_{\e})
=\liminf_{\e\to 0} \FUe
$
then, up to a subsequence, the following holds
\begin{enumerate}[label=\emph{\roman*)}]
\item $\displaystyle{\lim_{\e \to 0}}(\inf E_{\e,U})=\min E_0 $
\item
$\vue \to ((1\wedge x_3+1)\xoverline{\vect u},0)$ with $\xoverline{\vect u}\in H^1(\o,\R^2)$, where $
E_0(\xoverline{\vect u})=\min E_0
$.
\end{enumerate}

\label{thm:uniaxial}
\end{theorem}
Theorem~\ref{1710312258}
follows from matching a lower and an upper bound. 
As for the latter, we introduce a two-variable functional, construct a recovery sequence (pair) for displacements and the effective (thickness averaged) piecewise constant $Q$-tensors, further obtain the associated $\G\hbox{-}\limsup$ functional to approximate, by density, any (effective) $Q$-tensor in $L^2$.
We recover the full $\Gamma\hbox{-}$convergence result and the proof of Theorem \ref{1710312258} at the end of Section \ref{1711071646}.

\paragraph{The $\boldsymbol{\G\hbox{-}\liminf}$ inequality or the weak lower semicontinuity of $\boldsymbol{\FUe}$ } 

\begin{lemma}[$\G\hbox{-}\liminf$]
\label{lem:gliminfuniax}
Let $\vect u\in \mathcal V_0
$. For any sequence $(\vue)_{\e>0}\subset L^2(\O,\R^3): \vue \to \vect u$ strongly in $L^2(\Of, \R^3)$, we have
\[
    \liminf_{\e\to0}  \FUe( \vue)\geq E_0(\vect u) 
\]
where $\FUe$ and $E_0$ read as in Theorem~\ref{thm:uniaxial}.
\end{lemma}
\begin{proof}
This result follows
simply noticing that  ${\FUe} (\cdot) \geq {\FBe} (\cdot)  $
and arguing as in Lemma~\ref{lem:lowerboundbi}.

\end{proof}

\paragraph{The $\boldsymbol{\G-\limsup}$ inequality or the existence of a recovery sequence for $\boldsymbol{\FBe}$} 

The discussion of the $\Gamma$-limsup inequality can be split into \emph{i)} constructing a recovery sequence and \emph{ii)} matching the ansatz-independent lower energy bound with the upper bound computed on the recovery sequence.
The first step unveils the coupling mechanisms whilst the second requires careful energy estimates taking into account the three dimensional bulk microstructure and boundary layers at different scales. 
Because the energy functional is local and additive, we can match bounds term to term, separately, in the two layers.
We show that each target in the limit space $\left\{\vect v=(\xoverline {\vect v},0), \xoverline {\vect v} \in H^1(\o,\R^2) \right\}$ can be approximated by a suitably convergent, optimal, recovery sequence (cf. Corollary~\ref{cor:gupperlim} {and Theorem~\ref{thm:glimsupuniax}}). 
In order to accomplish this task we prove an intermediate relaxation result for a functional depending on a displacement and a thickness-invariant Q-tensor field (Theorem~\ref{thm:uniaxial}). 
The limit energy, as a consequence of the dimension reduction, yields a functional defined on a two-dimensional domain $\o$ where the optic tensor appears, as it is natural at first order, only through its thickness average. 
We establish the desired $\Gamma$-convergence result exploiting the following corollary.

\let\oldurs\urs
\let\oldQrs\Qrs
\renewcommand{\urs}{\vect v^k}
\renewcommand{\Qrs}{Q^k}

\begin{corollary}[$\G$-upper limit]\label{cor:gupperlim}
Under the assumptions of Theorem \ref{1710312258}
$$
    \Gamma\hbox{-}\limsup_{\e \to 0} {E_{U,\e}}(\vect u)\leq E_0(\vect u),
$$
where the $\Gamma\hbox{-}$lim sup is done in the strong $L^2(\Of,\R^3)$-topology.
\end{corollary}
\noindent Corollary~\ref{cor:gupperlim} is a consequence of the following $\G$-convergece theorem for a two-variable functional.
The proof of Corollary~\ref{cor:gupperlim} is postponed after the proof of Theorem~\ref{thm:glimsupuniax}.

\renewcommand{\urs}{\vect v^k}
\renewcommand{\Qrs}{Q^k}
\newcommand{\enfilm}[1]{I_{U,\e}(#1)\rfloor_{\Of}}
\newcommand{\ennema}[1]{I_{U,\e}(#1)\rfloor_{\Ob}}
Let us introduce an auxiliary two-variable functional, namely
for $(\vect u, Q)\in \mathcal{V}\times H^1(\Ob,\Qu)$ define
\begin{equation}
\label{eqn:twovarenergy}
\begin{aligned}
I_{U,\e}(\vect u,Q):=&
\frac{1}{2}\int_\Of \left\{ \tn{\ke{\vect u}}^2 + \frac{\laf}{2\muf} \tr^2(\ke{\vect u})   \right\} dx\\
    &\qquad+
    \frac{1}{2}\int_\Ob \left\{ \tn{\kebl{\vect u}-Q}^2 + \frac{\laf}{2\muf} \tr^2(\kebl{\vect u})
    +  \delta_\e^2 \left( \e^2|\nabla' Q|^2+\tn{\pt Q}^2  \right) \right\}  dx
\end{aligned}   
\end{equation}
extended $I_{U,\e}(\vect u,Q):=+\infty$ otherwise in $L^2(\O,\R^3)\times L^2(\Ob,\R^{3\times 3})$. 
We have the following theorem.

\begin{theorem}[$\G$-upper limit]
\label{thm:glimsupuniax}
For all $\vect u\in L^2(\O,\R^3)$, and $ \xoverline{Q}\in L^2(\O_b,\Qbb)$ constant along $x_3$, then
 \[
    \Gamma\hbox{-}\limsup_{\e \to 0} {I_{U,\e}}(\vect u, \xoverline Q)\leq I_0(\xoverline{\vect u}, \xoverline Q),
 \]
where ${I_{U,\e}}$ is defined in~\eqref{eqn:twovarenergy},
 the $\G-\limsup$ is computed for the strong $L^2(\Of,\R^3)$-topology in the variable $\vect u$ and the weak $L^2(\Ob,\R^{3\times 3})$-topology in the variable $\xoverline Q$, and $I_0$ reads
\[
    I_0({\vect u}, \xoverline Q):= \frac{1}{2} \int_\o \left(  |\eab(\xoverline{\vect u})|^2 + \frac{ \laf}{\laf + 2\muf} (\eaa(\xoverline{\vect u}))^2 
    + \frac{\lab}{\lab+2\mub} (\xoverline Q_{33})^2 + 2 \left|\frac{1}{2}\xoverline u_\alpha-\xoverline Q_{\alpha 3}\right|^2 + |\xoverline Q_{\alpha \beta}|^2 \right)dx'
\]
when $\vect u{=(\xoverline{\vect u}, 0)}\in \mathcal{V}_0$ and $\xoverline Q\in L^2(\o,\Qbb)$, 
and
$I_0(\vect u, Q):=+\infty$ otherwise in $L^2(\O,\R^3)\times L^2(\Ob,\R^{3\times 3})$.

\end{theorem}

\begin{proof}

The strategy consists in constructing recovery sequences in the film and nematic layers,  glued to yield an admissible recovery pair $({\vect {\hat v}^\e}, \widehat Q^\e)$ for the two-variable energy functional $I_{U,\e}(\vect u, Q)$, such that ${\vect {\hat v}^\e}\to \vect u = ((1\wedge x_3+1)\xoverline{\vect u}, 0)$ strongly in $L^2(\O,\R^3)$ and $\widehat Q^\e\wto \xoverline Q$ weakly in $L^2(\Ob,\Qbb)$.
We operate in the film and nematic layer separately. In the former, the sequence is identical to the one constructed in the proof of the limsup inequality for the biaxial model in Lemma~\ref{lem:upperboundbi}.
The main peculiarity here is the construction of a sequence of displacements and uniaxial tensors which are optimal for the mechanical energy and relax the nematic manifold at the same time.
The (diagonal restriction to $\Ob$ of the) sequence of displacements and uniaxial tensors are constructed 
such that
\renewcommand{\Qrs}{\widehat Q^\e}
\renewcommand{\urs}{\vect{\hat v}^\e}
\begin{multline}
    \limsup_{\e\to 0}\frac{1}{2}\int_{\Ob}
    \left( \e^2\delta_\e^2|\nabla'\Qrs|^2+\delta_\e^2|\partial_3 \Qrs|^2+|\kbl{\urs}-\Qrs|^2+\frac{\lab}{2\muf}\tr^2\kbl{\urs} \right)  dx
    \\
    \leq 
\frac{1}{2}\int_{\o}
\left( |\xoverline Q_{\alpha \beta}|^2+2\tn{\frac{\overline{u}_\alpha}{2}
-\xoverline Q_{\alpha 3}}^2+\frac{\lab}{\lab+2\mub}(\xoverline Q_{33})^2 \right) 
dx'.
\end{multline}
Note that
the upper trace of the recovery sequence in the nematic bonding layer has to coincide with the lower trace of the recovery sequence in the film which ultimately parametrises the deformation of the system.

\medskip

If $\vect{u}$ is not in the form $\vect u= (\xoverline{\vect u},0)$ with $\xoverline{\vect u}\in H^1(\o,\R^2), \xoverline{\vect u}(\cdot, -1)=0$ a.e. $x'\in \o$, and  if $\xoverline{Q}$ is not in $L^2(\Ob,\Qbb)$ with $\overline{Q}$ constant along the third direction, then the claim is trivial.
We assume the converse in what follows.

Let us construct the recovery sequence, starting with the film layer where the computation is analogous to the biaxial case, indeed 
\renewcommand{\urs}{\vect v^\e}
for every $\vect u=(\xoverline {\vect u},0 ), \xoverline {\vect u}\in H^1(\o, \R^2)$ and $\xoverline{\vect u}(\cdot, -1)=0$ a.e. $x'\in \o$, introducing $\hat h^\e, h^\e$ as in~\eqref{eqn:auxconverglimsup}, and define
\begin{equation}
\label{eqn:recseqfilm}
\urs(x', x_3)=(\overline u_{\alpha}(x'), {\e^2 (x_3\hat h^\e(x') + h^\e(x')}).
\end{equation}
We readily obtain
\begin{equation}
    \label{eqn:limsupunifilm}
    \begin{aligned}
&\limsup_{\e\to 0}\frac{1}{2}\int_\Of \left(\tn{\ke{\vect v^\e}}^2 + \frac{\laf}{2\muf} \tr^2(\ke{\vect v^\e})   \right) dx\\
&\qquad=\limsup_{\e\to 0} \frac{1}{2} \int_\Of \left\{ \left( h_\e^2 + \frac{1}{2} \e^2|x_3\nabla' \hat h^\e + \nabla' h^\e|^2 + |\eab(\xoverline{\vect u})|^2\right) + \frac{\laf}{2\muf} \left(\hat h^\e + \eaa(\xoverline{\vect u}) \right)^2 \right\} dx \\
&\qquad= \frac{1}{2} \int_\o \left(  |\eab(\xoverline{\vect u})|^2 + \frac{ \laf}{\laf + 2\muf} (\eaa(\xoverline{\vect u}))^2 \right)  dx.
\end{aligned}
\end{equation}

The construction in the nematic layer is more involved because it requires to consider the interaction between elastic and nematic microstructure.
We break down this argument in successive steps:
We first solve a local problem on an open set $A_i$ whose structure is dictated by the strong anisotropy given by the separation of scales, that is the $A_i$'s take the form of the Cartesian product of an arbitrary open and bounded subset $\o_i\subseteq\o$ by the open interval $(-1, 0)$ (Step 1). 
In each of the $A_i$'s we first relax a constant and biaxial tensor $Q$.
This is a fully three dimensional problem (see, for a sketch, Figure~\ref{fig:dimension_reduction}) where the relevant differential operator for the elastic problem is the rescaled strain which, effectively enforcing a strong anisotropy, selects weakly converging sequences of kinematically compatible $Q$-tensors.
Then (Step 2), we extend the construction to the entire nematic layer and, therein, we relax a piecewise constant $Q$ on the partition given by the $A_i$'s, taking care of gluing local 
constructions.
The construction obtained up to this point, consisting in a recovery sequence of $L^2$-piecewise constant order tensors,  has to be mollified to satisfy the uniaxiality constraint whilst ensuring control of its curvature energy (Step 3).
Finally (Step 4), we exploit the continuity properties of the functional to approximate, by density of the space of piecewise-continuous biaxial $Q$-tensors in $L^2$, any biaxial $Q$ in the limit space $L^2(\o,\Qbb)$.

\begin{figure}[tb]
    \centering
    \includegraphics[width=.9\textwidth]{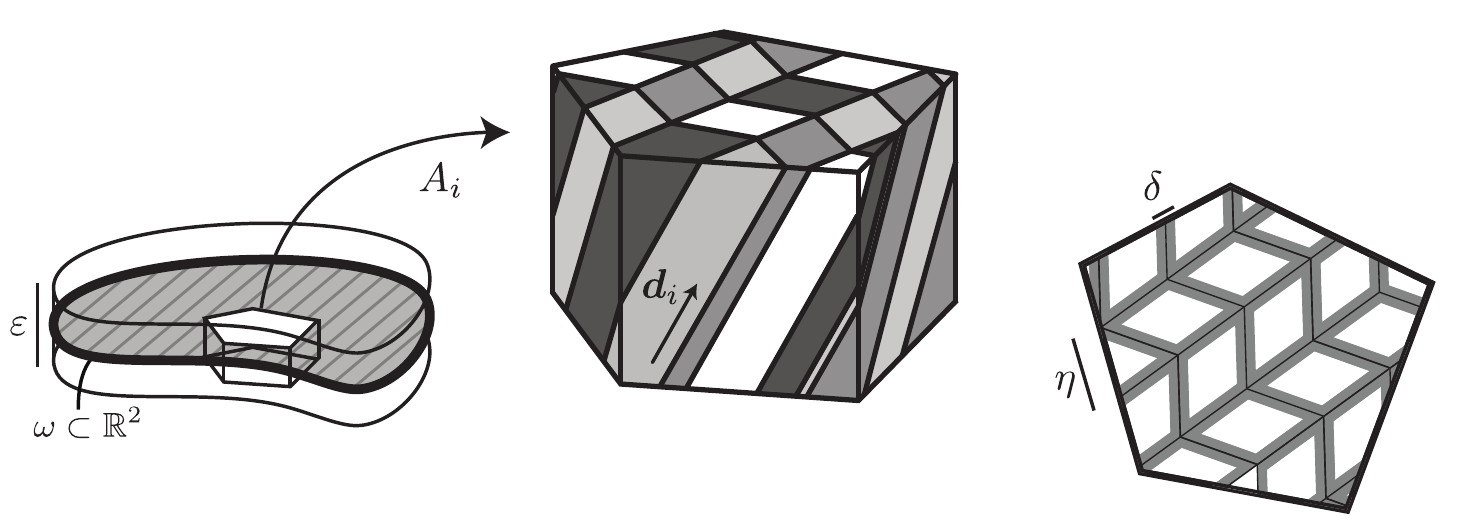}
    \caption{A sketch of the construction of the microstructure. On the left, the three dimensional domain, in the centre, a generic grain $A_i$ of thickness $\e$ with its optical microstructure $\Qrs$. Note that, in general and although the limit $Q$ is constant with respect to the thickness, the distinguished direction of the sequence of microstructures is not orthogonal to $\o$. On the right, a cross section of the microstructure, orthogonal to the distinguished direction $\vect d_i$ showing  microstructure size $\eta$ and boundary layer $\delta$.}
    \label{fig:dimension_reduction}
\end{figure}

\newcommand{\targetQ}{\xoverline Q}
\paragraph{Step 1. $\boldsymbol{\xoverline{Q}}$ constant (and biaxial)}
We localise our construction on some domain $\R^3\supset A_i:=\o_i\times(-1,0)$,  targeting a displacement $\vect u=((x_3+1)\xoverline {\vect u},0), \xoverline {\vect u}\in  H^1(\o_i, \R^2), \xoverline{\vect u}(\cdot, -1)=0$ a.e. $x'\in \o$ and an optic tensor $\targetQ\in \Qbb$, constant  and biaxial in $A_i$.
In this step, we drop the index $i$ from all sequences.
We exhibit a recovery pair 
$\urs \in \mathcal{V}$
and 
$\Qrs\in L^2(\Ob,\Qu)$, with $\urs\to ((x_3+1)\xoverline{\vect u},0)$ strongly in $L^2(\Ob, \R^3)$ and $\Qrs\wto \xoverline Q$ weakly in $L^2(\Ob,\R^{3\times 3})$ as follows.
We construct an oscillating optical microstructure with characteristic length scale $\eta$ and the associated accommodating displacements, denoted by $Q^\eta$ and  $\vect f^{\e,\eta}$, respectively.
Recall that the nematic stiffness imparts the smallest length scale to the system and bounds from below the size of the emergent microstructure which, in turn, is bounded from above by the thickness of the layer.
We hence have the order relation
\newcommand{\dir}{\vect n^{\eta}}
\renewcommand{\Qrs}{Q^\eta}
\[
    \delta_\e \ll \eta \ll \e.
\]
The three-dimensional construction  heavily draws from  \cite{cesana2011nematic}, it is reported in the Appendix and summarised as follows.
We apply Proposition~\ref{prop:recseq} (cf. {Appendix}) yielding a collection of open sets 
$\mathcal F_j^k$
such that $A_i=\bigcup_{k=1}^{N(\eta)}\mathcal F_j^k$ up to a set of measure zero (see Figure~\ref{fig:microstructure}). Further, there exist four matrices $Q_j\in \Qu$
and a sequence $Q^\eta\wto \xoverline Q\in L^2(A_i,\R^{3\times3})$ as $\eta\to 0$ such that
\begin{equation}
    Q^\eta(x)\bigr|_{\mathcal F_j^k}=Q_j,\quad j\in\{1,\dots,4\}.  
\end{equation}
The cardinality of the collection $N(\eta)$ is (at most) $\mathcal O(\eta^{-2})$ when the texture is columnar, otherwise $N(\eta)=\mathcal O(\eta^{-1})$ when the texture is striped.

Again, by Proposition~\ref{prop:recseq}, there exists $(\vect f^\eta)\subset W^{1, \infty}(A_i, \R^3)$
 such that $\vect f^\eta\to \xoverline Q.x$ uniformly on $\xoverline A_i$ as $\eta \to 0$ and $\nabla \vect f^\eta\wto Q$ in $L^2(A_i, \R^{3\times3})$.
The sequences $\vect f^\eta$ and $Q^\eta$ are crucially related by  $\frac{1}{2} \left( \nabla \vect f^\eta + (\nabla \vect f^\eta)^T \right)=Q^\eta \in\Qu \text{ a.e. in } A_i$.

Since $\Qrs$ is a piecewise constant uniaxial tensor field it can be lifted in $A_i$ in the sense that on each set $\mathcal F_j^k$ where the map $\Qrs$ is constant and equal to $Q_j$, by applying the spectral theorem there exists a unit vector $\vect n^{\eta}_j$ such that
\[
    \Qu \ni \Qrs_j=\vect n^{\eta}_j\otimes \vect n^{\eta}_j-\frac{1}{3} \id.
\]
We can therefore define the unit vector field $\vect n^\eta: A_i\mapsto S^2$ as
\begin{equation}
    \label{eqn:lifting}
  \vect n^{\eta}(x)= {\vect n^{\eta}_j}\text{ on } \mathcal{F}^{\eta}_j 
  \text{ for all }\eta.
\end{equation}

We now turn to the construction of the elastic microstructure.
The mechanical microstructure is conceived in such a way that $\kebl{\vect f^{\e,\eta}}=Q^{\eta}$
, hence cancellations (or elastic accommodation of optic texture) take place. 
\newcommand{\fke}{  \vect f^{ \e,\eta}}
Now, let us define $\vect f^{\e,\eta}$ as
\begin{eqnarray}\label{1711021740}
  \fke(x):=\phi^\e\circ \vect  f^\eta =
\begin{system}
 \eta f_\alpha(x'/(\eta\e), x_3/\eta)\\
 \eta {\e^2}  f_3(x'/(\eta\e), x_3/\eta )
\end{system} 
\end{eqnarray}
where we have introduced the anisotropic rescaling   via the map $ \phi^\e$ acting on the vector $\vect v$ by 
 modifying the spatial frequency in the plane and rescaling the vertical component 
 as follows
 \[
    \phi^\e \circ  \vect v = (v_\alpha(x'/\e, x_3), \e^2 v_3(x'/\e, x_3)).
\]
The mapping (\ref{1711021740}) exposes the coupling between the optical and the mechanical microstructures, characterised by their respective length scales $\eta$ and $\e$. Note that 
$\vect f_n$ defined in Proposition~\ref{prop:recseq} is equivalent to $\fke$ for $\e=1, \eta=1/n$.
By Proposition~\ref{prop:recseq}.(iii), the sequence $\vect f^{\eta}$ is such that $\nltwo[\vect f^\eta - \xoverline Q.x]{L^{2}}\leq C\eta$.
Then we have $\phi^\e \circ \vect f^\eta \to \phi^\e \circ(\targetQ.x) $ uniformly,  as $\eta\to 0$ for fixed $\e$ over $\xoverline A_i$.
Since we are interested in the limit as $\e\to 0$, we need to specify the relation between the two parameters $\e,\eta$. Owing to the order relation between length scales introduced above, we let
\begin{eqnarray}\label{eqn:etavseps}
\frac{\eta}{\e}\to 0
\end{eqnarray}
{}for $\eta\to 0$, as $\e\to 0$. In other words, we specify that the scale at which the optical microstructure occurs (measured by $\eta$) is finer than the thickness, this in turn allows for a genuinely three-dimensional mirostructure. Consequently, as both $\e,\eta$ vanish and (\ref{eqn:etavseps}) holds, it is possible to extract a diagonal sequence $\eta_\e:=\eta(\e)$ so that $\vect f^{\e,\eta} - \phi^\e \circ(\xoverline Q.x)$  uniformly converges to zero.

Finally, we define the sequence of vector maps

\renewcommand{\urs}{\vect v^{\e,\eta}}
\[
     \urs  = (x_3+1)\left(  \xoverline{\vect u} + {\e^2}  h^\e \vect  e_3 \right)  +\theta(x) \phi^\e \circ (  \vect  f^\eta - \targetQ .  x)
\]
or, with compact notation,
\[
     \vect v^{\e,\eta} =  \vect v^{\e} +\theta  \vect w^{\e,\eta},
\]
identifying
$ \vect v^{\e}:= (x_3+1)\left(  \xoverline{\vect u} + {\e^2}  h^\e \vect  e_3 \right)$
and
$\vect w^{\e,\eta}:= \phi^\e \circ (  \vect  f^\eta - \targetQ .  x)$.
Here we require
  $C^{\infty}_c({\o_i})\ni h^{\e}(x')\to \frac{2\mub}{\lab+2\mub} \xoverline Q_{33}\in L^2({\o_i})$ to satisfy optimality analogously to the biaxial case.
   The sequence above converges, for $\e,\eta\to0$ with $\eta/\e\to0$,  to $((x_3+1)\vect{\xoverline u},0) \in \mathcal V$ strongly in $L^2(\Ob,\R^3)$.
\newcommand{\Krho}{K^\rho_i}
In the expression above,
$\theta\in C^{\infty}_c(A_i)$ is a smooth scalar cut-off function which is required in order to
recover, for every $\eta$, homogeneous displacements at the boundary of $A_i$. It is identically equal to $1$ on $\Krho$ and $0\leq \theta(x)\leq 1$ on $A_i\setminus \Krho$, where $\Krho$ is the fixed compact set $\Krho:=\{x\in A_i:\operatorname{dist}(x,\partial A_i)\geq\rho\}$.
Note that $|\nabla\theta|\leq 1/\rho$ and 
when $\rho$ is small enough so that $K$ is actually non-empty, we have $
\operatorname{meas}(A_i\setminus \Krho)=\mathcal O(\rho)$.

The recovery sequence for the displacement field  is therefore the sum of two terms. The first, $ \vect v^{\e}$, recovers in the limit the optimal profile of the pure shear deformation of the nematic bonding layer.
Indeed, its planar components are fixed with respect to $\e$ and produce the affine deformation $(1+x_3)\xoverline u_\alpha$, matching the boundary condition on $\o\times \{-1\}$ and the upper trace of displacement of the film on $\o\times \{0\}$, hence ensuring continuity.
The out-of-plane component depends upon $\e$ via a smooth function which, in the limit, optimally accommodates the coupling between in-plane and out-of-plane deformations, in analogy to the film layer.
The second term, $ \rsms$, couples the elastic deformation with the nematic reorientation via formation of microstructure, at scales identified by  $\e$ and $\eta$.

\paragraph{Step 2. $\boldsymbol{\xoverline{Q}}$ piecewise constant (and biaxial).}
Here we generalize the construction of Step 1 to a biaxial tensor field $\xoverline Q(x)$ that is  piecewise constant in $\o$ and constant along $x_3$.
In general, there exists a partition of $\o$ consisting of a finite number $m$ of open and pairwise disjoint sets $\o_i$ such that
\[
\o=
\bigcup_{i=1}^m \o_i\cup N
\]
where the set $N\subset\R^2$ has measure zero.
Let 
 $Q_i:=\targetQ (x)|_{A_i}$. 
We now reintroduce the index $i$ referring to the generic set $A_i=\o_i\times (-1, 0)$,
 namely
\[
Q^{\eta}_i,\wto  Q_i \text{ weakly in } L^2(A_i,\R^{3\times 3}) \text{ as }\eta \to 0, \text{ and } \vect v^{\e,\eta}_i \to ((x_3+1)\overline{\vect u}, 0) \text{ uniformly in } \overline{A_i},
\]
where the latter convergence is intended for $\e\to 0$ provided~\eqref{eqn:etavseps} holds.
 Let us now extend the sequences to the entire domain inheriting the convergence properties displayed above. Define $Q^{\eta}:\Ob\mapsto \Qu$ and $\vect v^{\e,\eta}:\Ob\mapsto \R^3$ as follows
\[
Q^{\eta}(x):=Q^{\eta}_i \text{ on }A_i \text{ and }
   \vect v^{\e,\eta}(x):=  \vect v_i^{\e,\eta}(x),   \text{ on }A_i,
\]
so that $Q^{\eta}\wto \targetQ $ weakly in $L^2(\O_b,\R^{3\times 3})$. 
Note that it still follows that $\vect v^{\e,\eta}\in \mathcal V$
 for all $\e,\eta$, and the restriction $\vect v^{\e,\eta}(\cdot, 0)$ coincides with $\xoverline {\vect u}$.
Furthermore, $\vect v^{\e,\eta}$ converges strongly in $L^2(\Ob,\R^3)$ (indeed uniformly)
 to $((x_3+1)\xoverline{ \vect u}, 0)$ as $\eta, \e\to 0$ with $\eta/\e\to0$.

\paragraph{Step 3. Mollification of the sequence $\boldsymbol{\xoverline Q}^\eta$.}
Let us add the index $i$ to the sets defined in Step 1 
$\mathcal F_j^{k}(i)\subset A_i$
 to refer to the generic grain $A_i$.
Within each grain $A_i$
we regularise the sequence $Q^{\eta}(x)$ across all the interfaces 
between the
$\mathcal{F}_{j}^{k}(i)$'s
ensuring that the regularised sequence still ranges in  $\Qu$. This will guarantee control of the Dirichlet energy associated to the gradient of the tensor $Q$ and, at the same time, allows to satisfy the uniaxial constraint.
First, define compact sets with a very small `comfort zone' of size $\delta\ll \eta$ that approximate $\mathcal{F}^{k}_{j}(i)$ as $\delta\to 0$ as follows
\[
    \mathcal{F}^{k,\delta}_{j}(i):=\left\{ x\in \mathcal{F}^{k}_{j}(i):\operatorname{dist}
    \left( x, (\mathcal{F}^{k}_{j}(i))^c \right) \geq\delta,  j\in \{1, \dots, 4 \} \right\}.
    \]
Note that there exists $\delta$ small enough such that these sets are actually non-empty and they constitute a collection with the same cardinality of $\mathcal{F}_{j}^{k}(i)$ over each grain $A_i$.
\newcommand{\moll}{\psi_i^{\eta,\delta}}

Then, introduce the standard mollifiers $\moll$ 
which are identically equal to 1 over the sets $\mathcal{F}^{k,\delta}_{j}(i)$ and belong to $C^{\infty}_c(\mathcal{F}^{k}_{j}(i))$, for $j=1,\dots, 4$, to regularise our sequence of uniaxial tensors $\Qrs$.
Mechanical intuition suggests that the thickness of the boundary layer for $Q$ is dictated by the curvature energy which introduces the smallest length scale.
Finally, we convolute the `lifted' map $\dir$ (see Eq.~\eqref{eqn:lifting}) with $\moll$ yielding a smooth vector field which, in general, does not satisfy the requirement of being a unit vector field. In order to render the construction admissible, as already done e.g. in \cite{de-simone1993energy} and \cite{cesana2015effective}, we introduce the standard stereographic projection with pole $z$
\[
    \pi_z: S^2\mapsto\R^2.
\]
Let us define the (piecewise constant) map
\[  
 \vect q=
\begin{cases}
    \vect n^{\eta}, &\text{ on }
\bigcup_{i,j,k}\mathcal{F}^{k,\delta}_{j}(i) \\
 \vect n, &\text{ otherwise in } \R^3
\end{cases}
\]
for some $\vect n\in S^2$. 
Since $Q^{\eta}(x)$ has finite range we can use the following composition
\renewcommand{\dir}{{\vect n}^{\eta,\delta}}
\[
    {\vect n}^{\eta,\delta}:=\pi^{-1}_{s_i}\circ \left( \moll*(\pi_{s_i}\circ   \vect q) \right) \in C^{\infty}(\R^3,S^2),
\]
which allows us to reconstruct the uniaxial smooth tensor field on $\R^3$
\newcommand{\Qrssmooth}{{Q}^{\eta,\delta}}
\[
    \Qrssmooth(x):={\vect n}^{\eta,\delta}\otimes {\vect n}^{\eta,\delta}(x)-\frac{1}{3}\id
\]
further restricted to $A_i$. We now need to estimate the curvature energy associated to the gradient of $\Qrssmooth(x)$.
First of all, note that
$|\Qrssmooth|^2=|\nabla {\vect n}^{\eta,\delta}\otimes  {\vect n}^{\eta,\delta}+ {\vect n}^{\eta,\delta}\otimes \nabla {\vect n}^{\eta,\delta}|^2\leq C|\nabla  {\vect n}^{\eta,\delta} |^2$, for some constant $C$.
By repeatedly applying the chain rule we obtain (componentwise, here $l=1,2,3$)
\[
\partial_l {\vect n}^{\eta,\delta}=
\partial_l\left\{ \pi^{-1}_{s_i}\circ
\left( \moll*(\pi_{s_i}\circ  \vect q) \right)  \right\}=(\partial_l
\pi^{-1}_{s_i})\circ
\left( \moll*(\pi_{s_i}\circ   \vect q) \right) 
\cdot
\partial_l 
\left( \moll*(\pi_{s_i}\circ   \vect q) \right) \]
where
\[
    \partial_l \left( \moll*(\pi_{s_i}\circ   \vect q) \right) =
\partial_l  \moll* \left( \pi_{s_i}\circ   \vect q \right) 
\]
and $|\partial_l  \moll|\leq C\delta^{-1}$ over small transition layers of measure $\eta\delta$.
Therefore, since the stereographic projection is smooth, we obtain the estimate: 
$|\nabla \Qrssmooth|^2\leq C_i{\delta^{-2}}$, where the constant $C_i$ neither depends on $\delta$ nor or $\eta$ but only on $i$, that is 
on the grain and its optical microstructure. {This estimate holds, again, over small transition layers of measure $\eta\delta$.} 
As in~\cite{cesana2015effective} we have
$  \Qrssmooth=  {Q}_i^{\eta }$ over 
$\mathcal{F}^{k,\delta}_{j}(i)$ with $j\in \{1,\dots,4\}$, 
 and similarly for the gradient.

Let us compute the energy for a single grain $A_i$ starting by the curvature energy, splitting it into the bulk and boundary layer contribution,
we have,
\begin{multline}
\label{eqn:error1}
\frac{1}{2}\int_{A_i}\delta_\e^2 \left( \e^2|\nabla \Qrssmooth|^2+|\pt \Qrssmooth|^2 \right)  dx=\\
\delta_\e^2 \left( 
\frac{1}{2}\int_{{\bigcup_{j,k} \mathcal{F}_{j}^{k,\delta}(i) }}
\left( \e^2|\nabla' \Qrssmooth|^2 + |\pt \Qrssmooth|^2 \right)dx  + \frac{1}{2}\int_{ {A_i\setminus\bigcup_{j,k} \mathcal{F}_{j}^{k,\delta}(i)} }
\left( \e^2|\nabla' \Qrssmooth|^2 +|\pt \Qrssmooth|^2\right)dx \right) \\
\leq 0+C_i\delta_\e^2 \left( \e^2 + 1\right)\frac{1}{\delta^2}\times \eta\delta \times\frac{1}{\eta^2}.
\end{multline}
In the integral above, $j$ and $k$ range implicitly within $\{1, \dots, 4\}$ and $\{1, \dots, N(\eta)\}$.

\newcommand{\oscillrs}{{\vect{w}}^{\e,\eta}}
\newcommand{\argfoo}{{\vect{v}}^{\e}+\theta \vect{w}^{\e,\eta}}
\renewcommand{\ve}{\vect v^{\e}}

Now, we turn to the computation of the rescaled strains associated with the elastic recovery sequence, namely
$\kebl{\argfoo}$, before computing its energy.
The computation of the microstructure-free term $\kebl{\ve}$ is analogous to the biaxial case. 
Let us compute the scaled strains associated with oscillating terms, recalling that $\oscillrs= \fke - \phi^\e(\targetQ.x)$
\[
\begin{aligned}
    \kebl{\theta \oscillrs} &= 
 \begin{pmatrix}
        \e \nabla' \theta \odot \left( \oscillrs \right)' + \theta \kabbl{\oscillrs}& 
        \frac{1}{2} \left( {\e^{-1}}\pa \theta {w}^{\e,\eta}_3 + \pt \theta w^{\e, \eta}_\alpha\right) + \theta\ktabl{\oscillrs}    \\
        \frac{1}{2} \left( {\e^{-1}}\pa \theta {w}^{\e,\eta}_3 + \pt \theta w^{\e, \eta}_\alpha\right) + \theta\ktabl{\oscillrs}   & {\e^{-2}}\left( \pt \theta w^{\e, \eta}_3 + \theta \ett(\oscillrs) \right) \\
    \end{pmatrix}\\
    &=
     \begin{pmatrix}
        \e \nabla' \theta \odot \left( \oscillrs \right)' & 
        \frac{1}{2} \left( \e^{-1}\pa \theta w^{\e, \eta}_3 + \pt \theta w^{\e, \eta}_\alpha\right)   \\
        \frac{1}{2} \left( \e^{-1}\pa \theta w^{\e, \eta}_3 + \pt \theta w^{\e, \eta}_\alpha\right)  & \e^{-1}\left( \pt \theta {w}^{\e,\eta}_3 \right)
    \end{pmatrix}\\
    &\qquad+
\theta\underbrace{\begin{pmatrix}
        \kabbl{\oscillrs}&   \frac{1}{2}(\e^{-1}\pa  w^{\e, \eta}_3 +\pt  w^{\e, \eta}_\alpha)   \\
        \frac{1}{2}(\e^{-1}\pa  w^{\e, \eta}_3 +\pt  w^{\e, \eta}_\alpha) & \kttbl{\oscillrs}   \\
    \end{pmatrix}}_{=Q^\eta-\overline Q}.
\end{aligned}
\]
Note that $\vect w^{\eta,\e}$  depends on \emph{both} $\e$ and $\eta$ albeit the associated rescaled strain $\kebl{\oscillrs}$  depends \emph{only} on $\eta$. 
We now turn to the computation of the bulk energy of the single grain, decomposing $A_i=B_i \cup A_i\setminus B_i$ where $B_i$ reads
\[
  B_i:= K_i^\rho \cap \bigcup_{j, k} \left\{ \mathcal F_j^{k, \delta}(i), \quad j=1,\dots, 4,\quad k=1, \dots, N(\eta) \right\}.
\]
One may visualise (a section of) the three-dimensional domain $B_i$ as the white set in the leftmost cartoon in Figure~\ref{fig:dimension_reduction}.
We now turn to the computation of the bulk energy of the grain $A_i$
\begin{multline}
\label{eqn:energyestimateAi}
\frac{1}{2}\int_{A_i}
\left\{ \tn{\kebl{\urs}-\Qrssmooth }^2+\frac{\lab}{2\mub}(\tr \kebl{\urs})^2 \right\} dx=\\
\frac{1}{2}\int_{B_i}
\left\{ \tn{\kebl{\urs}-\Qrssmooth}^2+\frac{\lab}{2\mub}(\tr \kebl{\urs})^2  \right\}dx
 +
\frac{1}{2}\int_{A_i\setminus B_i}
\left\{  \tn{\kebl{\urs}-\Qrssmooth }^2+\frac{\lab}{2\mub}(\tr \kebl{\urs})^2  \right\}dx=
 \\
 =\frac{1}{2}\int_{B} \left\{ \tn{\kebl{\rsel} + \kebl{\fke}
- \targetQ - \Qrs}^2+
\frac{\lab}{2\mub}\left( h^\e+\e (x_3+1)\pa \ua(x', 0) \right) ^2  \right\}dx
\\+
\frac{1}{2}\int_{A_i\setminus B_i} \left\{ 
\tn{\kebl{\urs}-\Qrssmooth}^2+\frac{\lab}{2\mub}(\tr \kebl{\urs})^2 \right\} dx\\
 =\frac{1}{2}\int_{B_i} \left\{ \tn{\kebl{\rsel} + \Qrs
- \targetQ - \Qrs}^2+
\frac{\lab}{2\mub}\left( h^\e+\e (x_3+1)\pa \ua(x', 0) \right) ^2  \right\}dx
\\+
\frac{1}{2}\int_{A_i\setminus B_i} \left\{ 
\tn{\kebl{\urs}-\Qrssmooth}^2+\frac{\lab}{2\mub}(\tr \kebl{\urs})^2 \right\} dx.
\end{multline}
We now estimate the last integral in Equation~\eqref{eqn:energyestimateAi}. Remembering in what follows that $i$ is fixed, we have
\[
\frac{1}{2}\int_{A_i\setminus B_i} \left\{ 
\tn{\kebl{\urs}-\Qrssmooth}^2+\frac{\lab}{2\mub}(\tr \kebl{\urs})^2 \right\} dx
\leq
\frac{C_i}{2}\int_{A_i\setminus B_i} 
\left\{ \tn{\kebl{\urs}}^2 +
\tn{\Qrssmooth}^2 \right\} dx.
\]
Since the order tensor is always bounded, so is the second summand, namely
\[
\int_{A_i\setminus B_i} 
\tn{\Qrssmooth}^2 dx\leq C_i \left( \rho + \frac{\delta}{\eta} \right). 
\]
Instead, for the first summand there holds
\begin{multline}
\label{eqn:boundaryterms}
\int_{A_i\setminus B_i} 
\tn{\kebl{\urs}}^2 dx
\leq C_i\int_{A_i\setminus B_i}
\left\{ \tn{\kabe{\vect u}}^2
+ |\Qrs-\targetQ|^2  \right\}dx
 +\int_{A_i\setminus B_i}|h^{\e}_3|^2
dx\\
+\int_{A_i\setminus B_i}  
 \left(
\e^2|\nabla'\theta|^2 |\ors_\alpha|^2
+ {\e^{-1}}|\pa\theta|^2 |\ors_3|^2 
+ |\pt\theta|^2 |\ors_\alpha|^2 
+ |\pt\theta|^2 |{\e^{-2}}\ors_3|^2
\right) dx.
\end{multline}
The first two summands of the right hand side of~\eqref{eqn:boundaryterms} can be bounded as follows
\newcommand{\vu}{\vect u}
\begin{equation}
    \label{eqn:error2}
    \int_{A_i\setminus B_i}
\left\{ \tn{\kabe{\vect u}}^2
+ |\Qrs-\targetQ|^2  \right\}dx
 +\int_{A_i\setminus B_i}|h^{\e}_3|^2
dx
\leq
C_i\rho(\e^2+1+\norm[h^\e]{L^2(A_i\setminus B_i)}^2)\leq C'_i\rho.
\end{equation}
The line above holds because $|\kabe{\vect u}|^2$ is of order $\e^2$, $ |Q^{\eta}-\targetQ|^2$ is always uniformly bounded, and $h^\e$ is converging in $L^2$ (hence bounded therein).

To control the last contribution in~\eqref{eqn:boundaryterms}, {recall that $\oscillrs_3$ is of order $\eta\e^2$} and $\nabla \theta \equiv 0$ on $K_i^\rho$, implying the following  estimate (written in compact notation)
\begin{equation}
    \label{eqn:error3}
        \int_{A_i\setminus B_i} |\nabla\theta|^2|\oscillrs|^2dx\leq C_i \frac{\eta^2}{\rho},
\end{equation}
{the line above holds} because $|\pa \theta|^2$ and $|\pt \theta|^2$ are lesser or equal to $|\nabla\theta|^2\leq {\rho^{-2}}$
and $|\oscillrs|^2$, that represents the wrinkling texture for nematic relaxation,
is uniformly bounded by $\eta^2$. When $\e\to 0$ (and hence $\eta\to 0$), we shall be able to extract a subsequence guaranteeing control of this remainder term.

\renewcommand{\Qrs}{Q^{\eta}}

\paragraph{Step 4. Any $\boldsymbol{\xoverline{Q}\in L^2(\o,\Qbb)}$.}
We now reconstruct the domain $\Ob$. Remember, $\xoverline Q$ is a piecewise constant biaxial tensor defined over $\Ob$ invariant with respect to $x_3$ and $\vect u=(\xoverline{\vect u}, 0)\in \mathcal{V}$.
Let us define $Q^{\eta,\delta}\in H^1(\Ob, \Qu)$ and $\vect v^{\eta,\e}\in \mathcal V$ as follows
\begin{equation}
\begin{aligned}
\label{eqn:recseqnema}
       Q^{\eta,\delta}(x)&:=\Qrssmooth_i \text{ on }A_i, \text{ and}\\
     \vect v^{\eta,\e}(x)&:=  \vect v^{\eta,\e}_i(x)  \text{ on }A_i.
\end{aligned}
\end{equation}
Their convergence properties are inherited verbatim from those in the single grain, namely $Q^{\eta,\delta}\wto \xoverline Q\in L^2(\Ob, \Qbb)$ weakly in $L^2(\Ob, \R^{3\times 3})$ as $\eta\to 0, \delta/\eta\to 0$, and $\vect v^{\eta,\e}\to ((x_3+1)\xoverline {\vect u}, 0)$ uniformly in $\xoverline \Omega_b$ as $\e\to 0$ and $\eta/\e\to 0$.
Upon evaluation of the total energy of the nematic layer on the recovery sequence~\eqref{eqn:recseqnema} we get
\begin{multline}
\label{eqn:energyrecseqnema}
\frac{1}{2}\int_{\Ob} 
\left(
\e^2\delta_\e^2|\nabla' \Qrssmooth|^2+\delta_\e^2|\pt \Qrssmooth|^2+
\tn{\kebl{\urs}-\Qrssmooth}^2+\frac{\laf}{2\mub}\tr^2 \kebl{\urs} 
\right)
dx
\leq\\
\sum_i^m \frac{1}{2}\int_{B_i}\left(
\tn{\kebl{ \vect v^{\e} } +\Qrs-\targetQ -\Qrs}^2
+\frac{\laf}{2\mub}(\tr \kebl{\urs}^2 \right)dx
\left.
+
\sum_i^m \frac{C_i}{m} 
\left( 
\rho\frac{\delta_\e^2}{\delta\eta}
+
\rho
+
\frac{\eta^2}{\rho}
+\frac{\delta}{\eta}
\right) 
\right\},
\end{multline}
where we have used the estimates of Step 3 for each of the $A_i$'s.
Also recall that $\Qrs \equiv \Qrssmooth$ and $\nabla \Qrs\equiv 0$ on $B_i$, for all $i=1,\dots, m$.
By expanding the integral on right hand side above we get
\begin{multline}
\sum_i^m \frac{1}{2}\int_{B_i}\left(
\tn{\kebl{ \vect v^{\e} } +\targetQ}^2
+\frac{\lab}{2\mub}\left( h^\e+\e (x_3+1)\pa \xoverline u_\alpha\right) ^2 \right)dx\\
=
\sum_i^m \frac{1}{2}\int_{B_i}\left(
\tn{\e(x_3+1)\eab(\xoverline {\vect u}) - \targetQ_{\alpha\beta}}^2
+2\tn{\frac{1}{2}(\overline{ u}_\alpha+\e{(x_3+1)} \pa h^\e)- \targetQ_{\alpha 3}}^2
\right.dx\\
+\left.
(h^\e- \targetQ_{33})^2
+\frac{\lab}{2\mub}\left( h^\e+\e (x_3+1)\pa \xoverline u_\alpha\right) ^2\right)dx.
\end{multline}
Equation~\eqref{eqn:energyrecseqnema} now reads
\begin{multline}
\frac{1}{2}\int_{\Ob} 
\left(
\e^2\delta_\e^2|\nabla' \Qrssmooth|^2+\delta_\e^2|\pt \Qrssmooth|^2+
\tn{\kebl{\urs}-\Qrssmooth}^2+\frac{\laf}{2\mub}\tr^2 \kebl{\urs} 
\right)
dx
\leq
\\
\frac{1}{2}\int_{\Ob}\left(
\tn{\e(x_3+1)\eab(\xoverline {\vect u}) - \targetQ_{\alpha\beta}}^2
+2\tn{\frac{1}{2}(\overline{ u}_\alpha+\e{(x_3+1)} \pa h^\e)- \targetQ_{\alpha 3}}^2
\right.dx\\
+\left.
(h^\e- \targetQ_{33})^2
+\frac{\lab}{2\mub}\left( h^\e+\e (x_3+1)\pa \overline u_{\alpha}\right) ^2\right)dx
+
C
\left( 
\rho\frac{\delta_\e^2}{\delta\eta}
+
\rho
+
\frac{\eta^2}{\rho}
+\frac{\delta}{\eta}
\right) 
\end{multline}
where the domain of integration on the right hand side above has been enlarged {exploiting the fact that} the integrand is non-negative.
\newcommand{\vdiag}{\vect v^\e}
\newcommand{\Qdiag}{Q^\e}

We now extract a diagonal sequence 
$\eta=\eta(\e)$  and $\delta=\delta(\e)=\delta_\e$ such that the order relation $\delta(\e)\ll \eta(\e)\ll\e$ holds.
Consequently, we write $(\vdiag, \Qdiag):=(\vect v^{\e, \eta(\e)}, Q^{\eta(\e),\delta(\e)})$, and recalling the convergences established for $h^\e$ (see~\eqref{eqn:auxconverglimsup}), we finally compute the limit as $\e\to 0$ yielding
\begin{multline}\label{1711071441}
\limsup_{\e\to0}
\frac{1}{2}\int_{\Ob} 
\left(
\e^2\delta_\e^2|\nabla' \Qdiag|^2+\delta_\e^2|\pt \Qdiag|^2+
\tn{\kebl{\vdiag}-\Qdiag}^2+\frac{\laf}{2\mub}\tr^2 \kebl{\vdiag} 
\right)
dx\\
    \leq
        \frac{1}{2}\int_{\o}\left(  \tn{\targetQ_{\alpha\beta}}^2+
    2\tn{\frac{1}{2}\overline{ u}_\alpha- \targetQ_{\alpha 3}}^2
    + \frac{\lab}{\lab+2\mub}(\targetQ_{33})^2 \right)dx'.
\end{multline}
To establish the inequality above we have taken the limit as $\e\to 0$ which leaves a (small) remainder term depending on $\rho$. Since $\rho$ is arbitrary the estimate is, in fact, exact.
Note that the domain of integration is $\o\subset \R^2$ because the integrand is independent of $x_3$.

Now, let us glue the recovery sequences constructed in the two layers
\begin{equation}
    \hat {\vect v}^\e(x)=
    \begin{cases}
    \vect v^\e(x), &\text{ in }\Of \text{, see~\eqref{eqn:recseqfilm}}\\
    \vect v^{\eta(\e),\e}(x), &\text{ in }\Ob \text{, see~\eqref{eqn:recseqnema}}
    \end{cases}\quad\text{ and }  
    \quad
    \widehat Q^\e(x)=
    Q^{\eta(\e),\delta(\e)}(x), \text{ in }\Ob \text{, see~\eqref{eqn:recseqnema}}.
\end{equation}
The sequence $\hat {\vect v}^\e\in \mathcal V$ and  as $\e\to 0$ with $\eta/\e\to 0$ we have
\[
    \hat {\vect v}^\e(x)\to ((1\wedge x_3+1) \xoverline{\vect u}, 0) \text{ strongly in } L^2(\O, \R^3), \text{ and } \widehat Q^\e(x)\wto \xoverline Q\text{ weakly in } L^2(\Ob, \R^{3\times 3})
\]
where $\xoverline {\vect u}\in H^1(\o,\R^2)$ and $\xoverline Q\in L^2(\o,\Qbb)$ piecewise constant in $x'$ and constant in $x_3$. We then have from  \eqref{eqn:limsupunifilm} and  \eqref{1711071441}
\begin{multline}
\limsup_{\e\to 0}
I_{U,\varepsilon}(\hat {\vect v}^\e,    \widehat Q^\e)
\\
\leq
  \frac{1}{2} \int_\o \left(  |\eab(\xoverline{\vect u})|^2 + \frac{ \laf}{\laf + 2\muf} (\eaa(\xoverline{\vect u}))^2 \right)  dx'
  +
    \frac{1}{2}\int_{\o}\left(  \tn{\targetQ_{\alpha\beta}}^2+
    2\tn{\frac{1}{2}\xoverline{\vect u}- \targetQ_{\alpha 3}}^2
    + \frac{\lab}{\lab+2\mub}(\targetQ_{33})^2 \right)dx'.
\end{multline}
The inequality above holds at the level of the $\Gamma\hbox{-}\limsup$, in other words,
\begin{eqnarray}\label{1711071600}
\Gamma\hbox{-}  \limsup_{\e\to0}{I_{U,\e}}(\vect u, \xoverline Q)
    \leq
\frac{1}{2}\int_{\o}\left( 
 |\eab(\xoverline{\vect u})|^2 + \frac{ \laf}{\laf + 2\muf} (\eaa(\xoverline{\vect u}))^2+ \tn{\targetQ_{\alpha\beta}}^2+
    2\tn{\frac{1}{2}\xoverline{\vect u}- \targetQ_{\alpha 3}}^2
    + \frac{\lab}{\lab+2\mub}(\targetQ_{33})^2 \right)dx'.
\end{eqnarray}
In (\ref{1711071600}) the $\Gamma\hbox{-}$lim-sup is computed in the strong topology of $L^2(\Of,\R^3)$ for the vector $\vect u$ 
 and the weak topology of $L^2(\Ob,\R^{3\times 3})$  for $Q$. We remind that the weak topology is metrisable over closed and bounded balls in $L^2$ \cite[Chapter 8]{dal-maso1993an-introduction}, which is enough for the present case: indeed, 
sequences $(Q^{\e})$ range in $L^2(\Ob,\mathcal{Q}_B)$,

Remember Eq.~\eqref{1711071600} holds for any $\overline{Q}\in L^2(\o,\mathcal{Q}_B)$ and piecewise constant in $x'$. We need to generalise the inequality to any biaxial tensor in $L^2(\o,\Qbb)$. We do so by using the density of piecewise constant  biaxial tensors.
Indeed,   any $\overline{Q}\in L^2(\o,\mathcal{Q}_B)$  can be approximated in the strong topology by a sequence of piecewise constant  tensors in $L^2(\o,\mathcal{Q}_B)$ (it is enough to repeat the proof of~\cite[Proposition 3]{cesana2011nematic} over $\o\subset\R^2$). Since the right-hand side in (\ref{1711071600}) is continuous in the strong $L^2(\Ob,\R^{3\times 3})$-topology for the variable $Q$ and the $\Gamma\hbox{-}$lim-sup is lower semicontinuous in the weak $L^2(\Ob,\R^{3\times 3})$-topology  we have
that
(\ref{1711071600}) holds for any general $\overline{Q}\in L^2(\o,\R^3)$.

\end{proof}

\begin{proof}[Proof of Corollary~\ref{cor:gupperlim}]

Eq. (\ref{1711071600}) holds for any $\overline{Q}\in L^2(\o,\mathcal{Q}_B)$. By plugging $\xoverline Q^*$ (the unique minimiser
 of the right hand side of~\eqref{1711071600} for fixed $\xoverline{\vect u}\in H^1(\o,\R^2)$) 
we have

\begin{eqnarray}
\Gamma\hbox{-}  \limsup_{\e\to0} {I_{U,\e}}(\vect u, \xoverline Q^*)
    \leq
I_0(\xoverline{\vect u}, \xoverline Q^*)=\min_{\overline{Q}\in L^2(\o,\Qb)}I_0(\xoverline{\vect u},\overline{Q})=E_0(\overline{\vect u}).
\end{eqnarray}
Then, trivially we have
\begin{eqnarray}
\Gamma\hbox{-}  \limsup_{\e\to0}\left(\inf_{Q\in H^1(\Ob,\Qu)}{I_{U,\e}}(\vect u,  Q)\right)
    \leq\Gamma\hbox{-}  \limsup_{\e\to0}{{I_{U,\e}}(\vect u, \xoverline Q^*)}
\end{eqnarray}
where the topology is strong $L^2(\Of,\R^3)$, as required.

 \end{proof}

\paragraph{Proof of Theorem \ref{1710312258}} 
\label{par:proof_of_theorem_1710312258}

\begin{proof}
The $\Gamma$-convergence result is   a  direct consequence of Lemma~\ref{lem:gliminfuniax} and~\ref{cor:gupperlim}~\cite{dal-maso1993an-introduction}. 
Items $i)$ and $ii)$ follow because the sequence of functionals $E_{U,\varepsilon}$ is equi-coercive.
It is immediate to verify that $E_0$ admits a minimum in $H^1(\omega,\R^2)$ since it is coercive and and lower-semicontinuous in the weak $H^1$-topology by convexity.
\end{proof}

\begin{remark}

The $\Gamma\hbox{-}$convergence result is defined in the strong $L^2(\Omega_f,\R^3)$-topology of the displacement $\vect u$. Note the topology needs not be specified in the nematic layer: the kinematics as well as the optical nematic states in $\Ob$ are therefore dominated by the upper layer. 
\end{remark}

\section{Discussion} 
\label{sec:discussion}

We have studied the interaction between small scale microstructure and the macroscopic response of a composite NLCE/elastic membrane bilayer, a structure of interest for technological applications.
The existence of material small parameters singularly perturbs the three dimensional system leading to a rich competition of elastic and nematic phenomena at several different material length scales.
We have treated this scenario as a micro-macro problem based on a coupled optimisation at two separated, yet interacting, scales.
The length scales at which microstructure emerges, a product of energetic competition of fine scale oscillating energy minimisers, is small compared with the geometric scales of the structure ({this is } represented by the limit {scenario} $\eta/\e\to 0$).
According to this viewpoint, the structure attains its equilibrium minimising its total energy, while the nematic elastomer optimises its orientation, locally, subject to the elastic state determined at the macroscopic level.

Our result is two-fold.
On the one hand, as a consequence of dimension reduction, we establish an effective two-dimensional limit model \emph{i)} involving the displacement $\xoverline{\vect u}$ defined on the mid-plane which  accounts for the planar deformation of the film, and \emph{ii)} yields an effective `active foundation' term due to the interaction between shear deformations of the nematic layer, its optic texture and the deformation of the overlying film.
The asymptotic nematic energy density is representable via the (square of) a function involving material parameters $\lambda, \mu$ measuring the distance of a tensor $A(\xoverline{\vect u})$ representing shear deformations from the convex set $\Qbb$. 
Therefore the energy density can be computed exactly in the norm $\|\cdot\|_{\mu,\lambda}$ (defined in~\eqref{1711221146}).
Due to dimension reduction, every admissible order tensor is indeed a membrane tensor field
and furthermore, the optimal tensor  $\xoverline{Q}^*$  is determined by the membrane displacement.
On the other hand, we establish a relation between two different NLCE models.
Indeed, we have studied the NLCE in the context of the uniaxial (with fixed nematic order), and the biaxial (allowing variable order), theories.
Mathematically, the analysis of these two separate models built within the above frameworks leads to the same limit functional. However, from a physical point of view, despite the first (Frank) is the relevant modelling scenario for the bulk behaviour of a NLCE, the latter (De Gennes) emerges and is justified as a macroscopic, coarse-grained, limit. 
Indeed, while there is clear experimental evidence that uniaxial stretches may alter the orientation of the liquid crystal by aligning the molecules along the direction of maximum stretch, whether a macroscopic deformation may actually affect pointwise the order states of the liquid crystal is debated. Therefore, the direct coupling between mechanical strain and order tensor dictated by~\eqref{eqn:energydensity}-right may be accepted if $Q\in \Qu$, while it may seem too simplistic if $Q\in \Qbb$.
The asymptotic process establishes the relation between the two models in the sense that the former (Frank) asymptotically leads to the latter (De Gennes) due to relaxation and microstructure, justifying the biaxial tensor as a limit effective model.
The $\G$-convergence process determines peculiar characteristics of the emerging (and vanishing) microstructure, showing the richness of a genuinely three dimensional phenomenon due to the interaction between order states, mechanical strains, and geometric constraints.

A physical interpretation of the limit model hints that shear energy may be suppressed by nematic reorientation and production of fine-scale elastic features solicited by the overlying membrane, whereas, on the other hand, elastic deformation of the film can be triggered by multiphysical interaction with the liquid crystal.
The linearity of the model at hand is prone to model extension to account for thermal stresses as well as electric and magnetic fields which we have not considered explicitly.
Such extensions, the analysis of equilibrium solutions in simple yet representative cases, and the numerical computation in complex scenarios, are left to subsequent contributions.

\paragraph{Acknowledgments.}
PC is supported by JSPS Research Category   
Grant-in-Aid for Young Scientists (B) 16K21213. PC is grateful to the kind hospitality of ENSTA-Paris. PC holds an honorary appointment at La Trobe University and is a member of GNAMPA.
ALB wishes to thank La Trobe University for hospitality during a short visiting period.

\section{Appendix}\label{Appendix}

For the reader's convenience we report the definition of the recovery sequence of uniaxial tensors and displacements which is presented in  \cite[Theorem 1]{cesana2011nematic}. 
 In the following proposition, with some abuse of notation, we identify a constant biaxial tensor field $Q(x)\in L^2(U, \mathcal{Q}_B)$ with the matrix $Q\in \mathcal{Q}_B$ itself.

\begin{proposition}
\label{prop:recseq}
Given any biaxial matrix $Q\in \mathcal{Q}_B$ and any open and bounded set $U\subset\R^3$, 
the following holds true

\begin{itemize}

\item[i)] There exists a sequence of piecewise constant  tensors $(Q_{n})\subset L^{2}(U,\mathcal{Q}_U)$, $\forall n$, such that
\begin{eqnarray}\label{1710262330}
Q_{n}\rightharpoonup Q\quad \text{weakly}-L^2(U,\R^{3\times 3}) \textrm{ as } n \to\infty\nonumber.
\end{eqnarray}
\item[ii)] For every $n$ there exist four uniaxial matrices $Q_j\in\mathcal{Q}_U$, $j=1,\dots,4$
and a family of open sets $\mathcal{F}_j^{k}$ with $j=1,\dots, 4$ and $k=1,\dots, N(n)$ such that
$$
Q_n(x)\Bigr|_{\mathcal{F}^{k}_j}=Q_j \textrm{ a.e. in }U ,\quad j=1,\dots,4.
$$
and $U=\bigcup_{k}^{N(n)}\mathcal{F}_j^k$ up to a set of measure zero. Here $N(n)$ is the cardinality of the family of the sets $\mathcal{F}^k_j$.
\item[iii)] There exists a sequence
  $(\vect f_{n})\subset W^{1,\infty}(U,\R^3)$, $\forall n$ such that
$$
\vect f_{n}\to Q.x \textrm{ uniformly on } \overline{U}, \textrm{ as } n \to\infty,
$$
with $\|\vect f_n-Q.x\|_{L^{\infty}(U,\R^3)}\leq C/n$
and 
$$
\frac{\nabla \vect f_n+(\nabla \vect  f_n)^T}{2}(x)=Q_n(x) \textrm{ a.e. in } U,\,\,\forall n.
$$

\end{itemize}
\end{proposition}

\paragraph{Construction of $(Q_n)$, $(\vect f_n)$, and $\mathcal{F}_j^k$.}
First of all, it is not restrictive to assume $Q$ is a diagonal matrix parametrised in the form
\begin{displaymath}
 Q = \left(
\begin{array}{ccc}
a & 0 & 0 \\
0 & b & 0 \\
0 & 0 & c
\end{array} \right),
\end{displaymath}
with 
$-\frac{1}{3}\leq a\leq b\leq c\leq \frac{2}{3}$, $a+b+c=0$. Indeed, if $Q$ is not diagonal it is enough to apply the spectral theorem and operate with the correspoding diagonal matrix.
The construction now follows two paths depending on the parameter $a$.
\\
\quad
\\
\noindent $\bullet$ Case $a\neq-1/3$.
Let
\begin{eqnarray}\label{defT}
T:=\sqrt{ \frac{c+1/3}{a+1/3} }
\end{eqnarray}
and define the sets (see also Figure
$\ref{fig:microstructure}$-left for a sketch of the construction)
\begin{equation}\label{1710262252}
\begin{aligned}
\mathcal{F}_1  &:=\Bigl\{(x_1,x_2,x_3)\in \R^3: -\frac{1}{T} x_1< x_3< -\frac{1}{T}(x_1-T),
0< x_3<  1, -T< x_1< T, -1< x_2< 1 \Bigr\},  \\
\mathcal{F}_2 &:=\mathcal{F}_1-(T,0,0), \\
\mathcal{F}_3 &:=\Bigl\{(x_1,x_2,x_3)\in \R^3: \frac{1}{T}(x_1-T)< x_3< \frac{1}{T}
x_1,-1< x_3< 0,
-T< x_1< T, -1< x_2< 1 \Bigr\}, \\
\mathcal{F}_4 &:=\mathcal{F}_3-(T,0,0).
\end{aligned}  
\end{equation}
Define now the following matrices  
\begin{eqnarray}\label{1711062210}
 G_1 =\left(
\begin{array}{ccc}
a & 0 & 2G_{a,c}  \\
-2G_{a,b} & b  & -2G_{b,c} \\
0 & 0 & c
\end{array} \right),
\quad  G_2 =\left(
\begin{array}{ccc}
a & 0 & 2G_{a,c}  \\
 2G_{a,b} & b  & 2G_{b,c} \\
0 & 0 & c
\end{array} \right),\nonumber\\
 G_3 =\left(
\begin{array}{ccc}
a & 0 & -2G_{a,c}  \\
-2G_{a,b} & b  & 2G_{b,c} \\
0 & 0 & c
\end{array} \right),
\quad G_4 =\left(
\begin{array}{ccc}
a & 0 & -2G_{a,c}  \\
2G_{a,b} & b   & -2G_{b,c} \\
0 & 0 & c
\end{array} \right),
\end{eqnarray}
where the constants $G_{a,b}$, $G_{a,c}$, $G_{b,c}$ are
defined as follows
\begin{eqnarray}\label{1711062143}
G_{a,b}=\sqrt{a+\frac{1}{3}}\sqrt{b+\frac{1}{3}},
\quad G_{a,c}=\sqrt{a+\frac{1}{3}}\sqrt{c+\frac{1}{3}},
\quad G_{b,c}=\sqrt{b+\frac{1}{3}}\sqrt{c+\frac{1}{3}}.
\end{eqnarray}

 \noindent $\bullet$ Case $a=-1/3$.
 We define
\begin{equation}\label{1710262248}
\begin{aligned}
\mathcal{F}_1&:=\bigl\{(x_1,x_2,x_3)\in \R^3 : -1< x_1,-1<x_2< 1,  0< x_3<1 \bigr\},  \\
\mathcal{F}_2&:=\bigl\{(x_1,x_2,x_3)\in \R^3: -1< x_1,-1<x_2< 1,  -1< x_3<0 \bigr\}, \\
\mathcal{F}_{3}&:=\bigl\{(x_1,x_2,x_3)\in \R^3: -1< x_1,-1<x_2< 1,  -2< x_3<-1 \bigr\}, \\
\mathcal{F}_{4}&:=\bigl\{(x_1,x_2,x_3)\in \R^3 : -1< x_1,-1<x_2< 1,  1< x_3<2 \bigr\}.
\end{aligned}
\end{equation}
Introduce the following matrices
\begin{eqnarray}\label{1711062212}
 G_1 =G_3=\left(
\begin{array}{ccc}
-\frac{1}{3} & 0 & 0  \\
0 & b  & -2G_{b,c} \\
0 & 0 & c
\end{array} \right),
\,\, G_2 =G_4=\left(
\begin{array}{ccc}
-\frac{1}{3} & 0 & 0  \\
0 & b  & 2G_{b,c} \\
0 & 0 & c
\end{array} \right),
\end{eqnarray}
where  $G_{b,c}=\sqrt{b+\frac{1}{3}}\sqrt{c+\frac{1}{3}}$.

Now what follows holds for both cases.
We define the tensor
field  
\begin{displaymath}
H(x) :=\left\{ \begin{array}{ccc}
  G_1   &\textrm{on } &\mathcal{F}_1 \\
  G_2  &\textrm{on } &\mathcal{F}_2  \\
    G_3   &\textrm{on } &\mathcal{F}_3 \\
   G_4   &\textrm{on } &\mathcal{F}_4 , \\
\end{array} \right.
\end{displaymath}
 where the four matrices $G_j$ and sets $\mathcal{F}_j$ are either the ones defined in Eqs. (\ref{1711062210}) and (\ref{1710262252}) (when $a\neq -1/3$) or Eqs. (\ref{1711062212}) and (\ref{1710262248}) (when $a=-1/3$) respectively.
By construction, the matrices $G_j$  are kinematically compatible
as the Hadamard jump conditions ~\cite{muller1998variational,dacorogna2007direct}  are verified across   the interfaces   separating them (see Fig. \ref{fig:microstructure}). 
We then define
$$
Q_j:=\frac{G_j+G_j^T}{2},\quad  j=1,\dots, 4.
$$
It follows that
\begin{eqnarray}\label{1710251803}
Q=\frac{1}{4}\sum_j^4 Q_j=\frac{1}{4}\sum_j^4 G_j,
\end{eqnarray}
with $Q_j\in\mathcal{Q}_U$ since
\begin{eqnarray}\label{1710251804}
\textrm{spectrum}\left(\frac{G_j+G_j^T}{2}\right)=\left \{-\frac{1}{3},-\frac{1}{3},\frac{2}{3}\right\},\quad j=1,\dots,4.
\end{eqnarray}
We define   the open set
$\mathcal{D}=\Bigl(\overline{\bigcup_j^4 \mathcal{F}_j}\Bigr)^{\circ}$ (i.e., the interior of the closure of the union of the sets $\mathcal F_j$).

\begin{figure}[h!]
\centering%
\includegraphics[width=.8\textwidth]{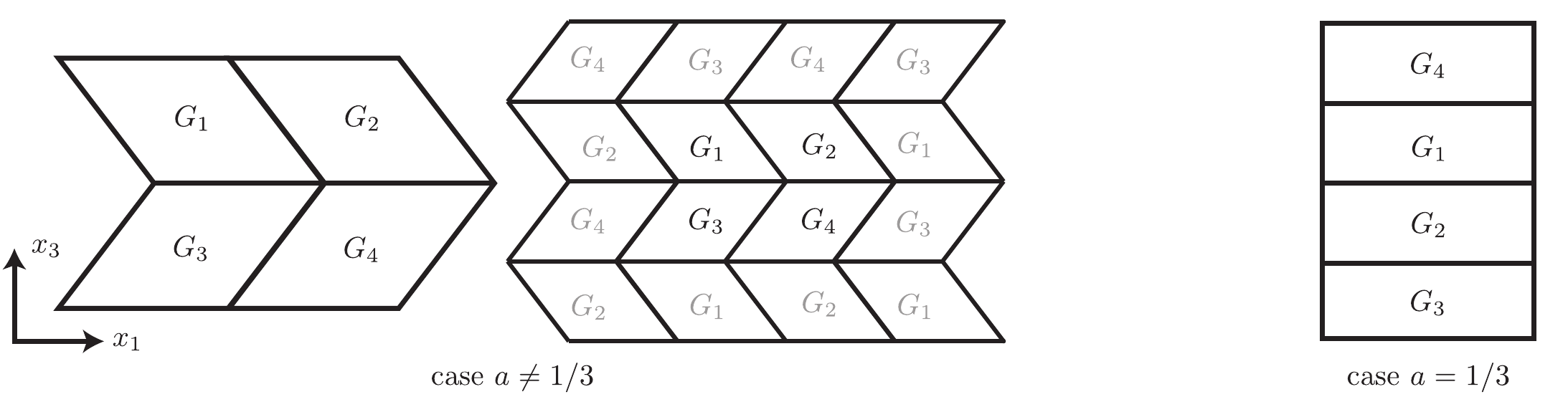}%
\caption{
The four matrices $G_j$ defined over the domain $\mathcal{D}$, case $a\neq -1/3$ (left panel, at different frequencies), case $a= -1/3$ (right panel). When $a=b=c=0$ the acute angles in the rhomboids measure $\pi/4$. 
}
\label{fig:microstructure}
\end{figure}

\noindent
Denote with $\widetilde{H}(x)$ the periodic
extension of $H(x)$ in $\R^3$ and define $\forall
\,n\in\mathbb{N}$
\begin{eqnarray}
 F_{n}(x):=\widetilde{H}(n\,x_1 ,n\,x_2,n\,x_3),\nonumber
\end{eqnarray}
\begin{eqnarray}\label{defQn}
 Q_{n}(x):=\frac{ F_n+F_n^T }{2}(x).
\end{eqnarray}
We denote with $\mathcal{F}_j^{k}$, where $j=1,\dots, 4$ and $k=1,\dots, N(n)$ the rescaled copies of the sets $\mathcal{F}_j$ defined in 
(\ref{1710262252}) and (\ref{1710262248}). There follows
\begin{displaymath}\label{1711062300}
F_n(x) =\left\{ \begin{array}{ccc}  
 G_1  &  \textrm{on } \mathcal{F}_1^k \\
   G_2 &  \textrm{on } \mathcal{F}_2^k\\
   G_3 &  \textrm{on } \mathcal{F}_{3}^k\\
 G_4  &  \textrm{on } \mathcal{F}_{4}^k,
\end{array} \right.\qquad
Q_n(x) =\left\{ \begin{array}{ccc}  
 Q_1  &  \textrm{on } \mathcal{F}_1^k \\
   Q_2 &  \textrm{on } \mathcal{F}_2^k\\
   Q_3 &  \textrm{on } \mathcal{F}_{3}^k\\
 Q_4  &  \textrm{on } \mathcal{F}_{4}^k,
\end{array} \right.
\end{displaymath}
for $k=1,\dots,N(n)$.
Thanks to (\ref{1710251803}) 
we have, for every open set $U\subset\R^3$
$$
 F_{n}(x)\stackrel{}{\rightharpoonup} Q\,\,\,\textrm{and }
 \,Q_{n}(x)\stackrel{}{\rightharpoonup} Q\,\,
\textrm{weakly in }L^{2}(U,\R^{3\times 3})\textrm{ as }n\to+\infty
$$
thus ensuring $i)$ and $ii)$.

To show $iii)$, remember the
construction of the matrices $G_j$ over the sets $\mathcal{F}_j^k$ is  kinematically compatible.
In other words, for every $n\in \N$, there exists a vector field $\vect f_n:U\to\R^3$ such that $\nabla \vect f_n(x)=F_n(x)$ 
and for which it is easy to prove that $\| \vect f_n(x)-Q .x\|_{L^{\infty}(U,\R^{3})}\leq C/n$ and therefore $iii)$ follows.

\vspace{6em}
\begin{multicols}{2}
\noindent Pierluigi Cesana \\
Institute of Mathematics for Industry, \\
Kyushu University, Fukuoka, Japan
\\ 
and\\
Department of Mathematics and Statistics,\\
 La Trobe University, Australia

\noindent e-mail: \url{cesana@math.kyushu-u.ac.jp}

\columnbreak

\noindent Andrés A. León Baldelli. \\
Institute of Mechanical Sciences and Industrial Applications\\
CNRS UMR 9219\\
\noindent e-mail: \url{leon.baldelli@cnrs.fr}

\end{multicols}

\printbibliography

\end{document}